\documentclass[final,onefignum,onetabnum]{siamonline171218}

\usepackage{amsfonts}
\usepackage{amssymb}
\usepackage{amsmath}
\usepackage{amsopn}
\usepackage{algorithmic}
\usepackage{booktabs}
\usepackage{makecell}
\usepackage{wrapfig}
\usepackage{tikz}

\usetikzlibrary{decorations.markings,decorations.pathmorphing}

\usepackage{enumitem}
\setlist[enumerate]{leftmargin=.8in}
\setlist[itemize]{leftmargin=.8in}

\newsiamremark{example}{Example}
\newsiamremark{remark}{Remark}
\newsiamremark{problem}{Problem Statement}

\DeclareMathOperator{\conv}{conv}

\newcommand{\C}{\mathbb{C}}

\newcommand{\R}{\mathbb{R}}
\newcommand{\Z}{\mathbb{Z}}
\newcommand{\imag}{\mathbf{i}}

\newcommand{\edges}{\mathcal{E}}

\newcommand{\boldzero}{\mathbf{0}}
\newcommand{\boldx}{\mathbf{x}}
\newcommand{\boldy}{\mathbf{y}}
\newcommand{\boldv}{\mathbf{v}}
\newcommand{\bolde}{\mathbf{e}}

\newcommand{\boldlambda}{\boldsymbol{\lambda}}
\newcommand{\apbound}{\ensuremath{N \binom{N-1}{\lfloor (N-1) / 2 \rfloor}}}

\headers{A toric deformation method for solving Kuramoto equations}{T.R. Chen and R. Davis}

\title{%
    A toric deformation method for solving Kuramoto equations on cycle networks
    \thanks{
        \funding{This work was funded, in part, by the AMS Simons Travel Grant,
        Auburn University at Montgomery Research Grant-in-Aid Program,
        and NSF under award number DMS-1923099 and DMS-1922998.}
    }
}

\author{Tianran Chen%
    \thanks{Department of Mathematics,
        Auburn University Montgomery, Montgomery, AL USA
        (\email{ti@nranchen.org}, \url{http://www.tianranchen.org/}).
    }
    \and Robert Davis%
    \thanks{Department of Mathematics, 
        Colgate University,
        Hamilton, NY, USA
        (\email{rdavis@colgate.edu}).
    }
}

\begin{document}

\maketitle

\begin{abstract}
    The study of frequency synchronization configurations
    in Kuramoto models is a ubiquitous mathematical problem 
    that has found applications in many seemingly independent fields.
    In this paper, we focus on networks whose underlying graph are cycle graphs.
    Based on the recent result on the upper bound of the 
    frequency synchronization configurations in this context,
    we propose a toric deformation homotopy method for locating all
    frequency synchronization configurations with complexity that is linear
    in this upper bound.
    Basing on the polyhedral homotopy method, this homotopy induces 
    a deformation of the set of the synchronization configurations into 
    a collection of toric varieties,
    yet our method has the distinct advantages of 
    avoiding the costly step of computing mixed cells
    and using special starting systems that can be solved in linear time.
    We also explore the important consequences of this homotopy method
    in the context of directed acyclic decomposition of Kuramoto networks
    and tropical stable intersection points for Kuramoto equations.
\end{abstract}



\section{Introduction}
A network of oscillators is a set of objects, varying between two states, 
that can influence one another.
A network of $N=n+1$ oscillators can be modeled
by a weighted graph $G = (V,E,K)$ with
vertices $V = \{0,\dots,n\}$ representing the oscillators,
edges $E$ representing the connections among the oscillators, 
and weights $K = \{k_{ij}\}$ representing the \emph{coupling strength} along the edges.
Each oscillator $i$ has its natural frequency $\omega_i$
so that $\frac{d\theta_i}{dt} = \omega_i$ 
if the oscillator is disconnected from the network.
In a network, however, oscillators influence one another and the dynamics
can be described by the generalized Kuramoto model~\cite{Kuramoto1975}
with phase shift given by the differential equations
\begin{equation}
    \frac{d \theta_i}{dt} =
    \omega_i -
    \sum_{j \in \mathcal{N}_G(i)} k_{ij} \sin(\theta_{i}-\theta_{j} + \delta_{ij}),
    \quad \text{ for } i = 0,\dots,n,
    \label{equ:kuramoto-ode}
\end{equation}
where each $\theta_i \in [0,2\pi)$ is the phase angle that describes the status
of the $i$-th oscillator,
$\mathcal{N}_G(i)$ is the set of its neighbors,
and the constant $\delta_{ij}$ quantifies the phase shift between oscillators $i$ and $j$.
Here we allow \emph{non-uniform coupling strength}
($k_{ij}$'s may not be identical)
and potential \emph{phase shift} ($\delta_{ij}$'s may be nonzero).
A fundamental mathematical problem in the study of Kuramoto model
as well as the behavior of coupled oscillators is the occurrence of
\emph{synchronization}.
Among many different notions of synchronization,
this paper focuses only on \emph{frequency synchronization}, 
which occurs when the competing forces of 
oscillators to stay with their natural frequency 
and the influence of their neighbors reach equilibrium for all oscillators 
and they are all tuned to the same frequency
and hence $\frac{d\theta_i}{dt} = c$ for a common constant $c$ for all $i$.
They are precisely the solutions to the system of nonlinear equations
\begin{equation}\label{equ:kuramoto-sin}
    \omega_i - 
    \sum_{j \in \mathcal{N}_G(i)} k_{ij} 
    \sin(\theta_{i}-\theta_{j} + \delta_{ij}) = c,
    \quad \text{ for } i = 0,\dots,n.
\end{equation}

In this paper, we focus on the cases where the underlying graph is a cycle,
i.e., the set of edges $E$ of $G$ consists of
$\{0,1\}, \{1,2\}, \dots, \{n-1,n\}, \{n,0\}$.
In~\cite{ChenDavisMehta2018}, an upper bound on the total number of 
isolated solutions the \emph{synchronization equations}~\eqref{equ:kuramoto-sin} 
has is shown to be $\apbound$
using the theory of birationally invariant intersection index.
Indeed, this upper bound is generically sharp for a complexified version of~\eqref{equ:kuramoto-sin}.
It is then natural to ask if there exists an algorithm that can locate 
all solutions of~\eqref{equ:kuramoto-sin} with a complexity that is
linear in this solution bound $\apbound$.
This is the main topic that this paper addresses.

The primary contribution of this paper is the development of a homotopy method
in the spirit of polyhedral homotopy that will find \emph{all} 
isolated solutions of~\eqref{equ:kuramoto-sin}.
The total number of homotopy paths to be tracked with this method
is exactly the solution bound $\apbound$.
Yet, this method offers significant advantages over a direct application of
polyhedral homotopy via the following features:
\begin{itemize}
    \item 
        our method avoids the costly step of computing mixed volume/cells;
        
    \item
        our method does not require solving binomial systems,
        and the starting systems can be solved in $O(N^2)$ time in serial or
        $O(N)$ time in parallel;
    \item 
        our method uses integer liftings of $\{0,1,2\}$ and hence avoids the well known
        numerical instability caused by random liftings.
\end{itemize}

The secondary contribution is an explicit description of a regular
unimodular triangulation of the adjacency polytope which significantly
strengthens the previous volume and facet description results~\cite{ChenDavisMehta2018}
and may shed new light on closely related constructions 
such as ``symmetric edge polytopes''~\cite{DelucchiHoessly2016Fundamental,SymmetricEdge2019,InterlacingEhrhart,SymmetricConfigurations,SmoothFanoEhrhart}.

The tertiary contribution is our significant refinement for the 
direct acyclic decomposition scheme proposed in~\cite{Chen2019Directed} 
for cycle graphs.
This refined scheme is capable of reducing a network into simplest subnetworks
known as primitive subnetworks for which frequency synchronization configurations
can be computed directly and efficiently.
Finally, we provide an interpretation of our result in terms of
tropical algebraic geometry as well as the equivalence of three
rather different perspective to the Kuramoto equations.

The paper is organized as follows. 
In~\cref{sec:kuramoto} we briefly review the Kuramoto model and Kuramoto equations.
\Cref{sec:formulation} describes a complex algebraic formulation of the Kuramoto equations
as a system of rational equations over the complex algebraic torus $(\C^*)^n$.
Recent results on the generic root count of the algebraic Kuramoto equations,
known as the adjacency polytope bound, is reviewed in~\cref{sec:ap}.
We strengthen this result by describing explicit formula
for a regular unimodular triangulation of the adjacency polytope
in~\cref{sec:triangulation}.
The resulting algorithm is outlined in~\cref{sec:algo},
and \cref{sec:implementation} describe the software implementation.
Basing on this triangulation, 
we develop our homotopy method in~\cref{sec:homotopy}.
In~\cref{sec:interpretation}
we interpret our results in the broader context,
and the conclusion follows in~\cref{sec:conclusion}.

\section{Kuramoto model and Kuramoto equations}\label{sec:kuramoto}

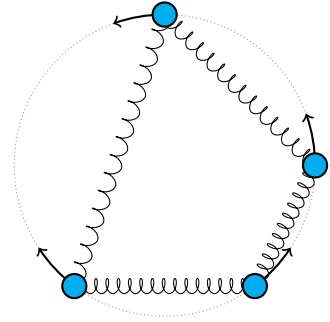
\begin{wrapfigure}[11]{r}{0.30\textwidth}
    \centering
    \begin{tikzpicture}[scale=0.4]
        \draw[densely dotted,color=gray] (0,0) circle(5.0);
        \draw[decoration={segment length=2.0mm,amplitude=1mm,coil},decorate] (5,0) -- (0,5);
        \draw[decoration={segment length=2.5mm,amplitude=1mm,coil},decorate] (0,5) -- (-3,-4);
        \draw[decoration={segment length=1.5mm,amplitude=1mm,coil},decorate] (5,0) -- (3,-4);
        \draw[decoration={segment length=1.5mm,amplitude=1mm,coil},decorate] (3,-4) -- (-3,-4);
        \draw[->,thick,color=black] ( 5, 0) arc[radius=5, start angle=0  , end angle=20];
        \draw[->,thick,color=black] ( 0, 5) arc[radius=5, start angle=90 , end angle=110];
        \draw[->,thick,color=black] (-3,-4) arc[radius=5, start angle=233, end angle=213];
        \draw[->,thick,color=black] ( 3,-4) arc[radius=5, start angle=307, end angle=327];
        \filldraw[thick,fill=cyan]  ( 5, 0) circle(0.4); 
        \filldraw[thick,fill=cyan]  ( 0, 5) circle(0.4); 
        \filldraw[thick,fill=cyan]  ( 3,-4) circle(0.4); 
        \filldraw[thick,fill=cyan]  (-3,-4) circle(0.4); 
    \end{tikzpicture}
    \caption{A spring network}\label{fig:spring-network}
\end{wrapfigure}
A simple mechanical analog of the coupled oscillator model~\eqref{equ:kuramoto-ode}
is a spring network, shown in~\cref{fig:spring-network}, that consists of
a set of weightless particles constrained to move on the unit circle without
friction or collision~\cite{dorfler_synchronization_2014,Kuramoto1975,Kuramoto2012}.
Real numbers $k_{ij} = k_{ji}$ characterizing the stiffness of the springs
connecting particles $i$ and $j$ are known as \emph{coupling strength},
and $\frac{d\theta_i}{dt}$ represents the angular velocity of the $i$-th particle.
In some applications, oscillators may be influenced by shifted phases of neighboring oscillators,
and such phase shift are characterized by $\{ \delta_{ij} \}$.

An important class of special configurations in which the angular velocity of
\emph{all} particles can become perfectly aligned are known as
\emph{frequency synchronization}.
That is, $\frac{d\theta_i}{dt} = c$ for $i=0,\dots,n$ and a constant $c$.
By adopting a rotational frame of reference, we can always assume $c = 0$.
That is, frequency synchronization configurations are equivalent to equilibria
of the ordinary differential equations~\eqref{equ:kuramoto-ode}.
As a result of the symmetry assumption that $k_{ij} = k_{ji}$, 
the $n+1$ equilibrium equations must sum to zero.
This allows the elimination of one of the equations,
producing the system of $n$ equations in $n$ unknowns
\begin{equation*}
    \omega_i - 
    \sum_{j \in \mathcal{N}_G(i)} k_{ij} 
    \sin(\theta_{i} - \theta_{j} + \delta_{ij}) = c,
    \quad \text{ for } i = 1,\dots,n.
\end{equation*}
Despite its mechanical origin, the above frequency synchronization system
naturally appears in a long list of seemingly unrelated fields,
including chemistry, electrical engineering, biology, and computer security.
We refer to~\cite{dorfler_synchronization_2014} for a detailed list.
The problem of find some or all frequency synchronization configurations
has been an active research topic in the recent decades
\cite{Casetti:June2003:0022-4715:1091,delabays2016multistability,delabays2017multistability,Hughes:2012hg,hughes2014inversion,kastner2011stationary,manik2016cycle,Mehta2015Algebraic,Mehta:2013iea,mehta2011stationary,Nerattini:2012pi,ochab2010synchronization,xi2017synchronization,xin2016analytical}

\section{Complex algebraic formulation of Kuramoto equations}\label{sec:formulation}

With proper complexification, the Kuramoto equations~\eqref{equ:kuramoto-sin}  can be 
reformulated as a system of Laurent polynomials over the algebraic torus $(\C^*)^N$.
This form is crucial in applying the root counting and homotopy continuation theory.
Using the identity
$\sin(\theta_{i} - \theta_{j} + \delta_{ij}) = \frac{1}{2\imag}
(e^{ \imag(\theta_{i} - \theta_{j} + \delta_{ij})} - e^{-\imag(\theta_{i} - \theta_{j} + \delta_{ij}}))$
where $\imag = \sqrt{-1}$,~\eqref{equ:kuramoto-sin} can be transformed into
\begin{equation*}
    c_{i} -
    \sum_{j \in \mathcal{N}_G(i)}
    \frac{k_{ij}}{2\imag} (
        e^{ \imag \theta_{i}} e^{-\imag \theta_{j}} e^{ \imag \delta_{ij}} -
        e^{-\imag \theta_{i}} e^{ \imag \theta_{j}} e^{-\imag \delta_{ij}} 
    ) = 0
    \quad \text{ for } i = 1,\dots,N.
\end{equation*}
With the substitution $x_{i} := e^{\imag \theta_{i}}$ for $i = 1,\dots,n$,
we obtain the system of rational equations
\begin{equation}\label{equ:kuramoto-rat}
    F_{G,i}(x_1,\dots,x_n) = c_{i} - \sum_{j \in \mathcal{N}_G(i)} 
        a_{ij} \frac{x_i}{x_j} - 
        b_{ij} \frac{x_j}{x_i}
    = 0
    \quad \text{ for } i = 1,\dots,n
\end{equation}
where 
$a_{ij} = \frac{k_{ij}}{2\imag} e^{ \imag \delta_{ij}}$,
$b_{ij} = \frac{k_{ij}}{2\imag} e^{-\imag \delta_{ij}}$,
and $x_0 = 1$.
$F_G = (F_{G,1},\dots,F_{G,n})$ is a system of $n$ Laurent polynomial
equations  in the $n$ complex variables $\mathbf{x} = (x_1,\dots,x_n)$.
In the following, it is referred to as the (algebraic) system of
\emph{synchronization equations} for a Kuramoto model, 
or simply a \emph{synchronization system}.
Since $x_i$'s appear in the denominators,
$F_G$ is only defined on the algebraic torus ${(\C^*)}^n = {(\C \setminus \{0\})}^n$.
Each equivalence class of real solutions of~\eqref{equ:kuramoto-sin}
(modulo translations by multiples of $2\pi$)
corresponds to a single solution of~\eqref{equ:kuramoto-rat} in ${(\C^*)}^n$.

If we consider $F_G$ as a column vector, then for any nonsingular  $n \times n$ matrix $R$, 
the systems $R \cdot F_{G}$ and $F_{G}$ 
have the same zero set.
Therefore in the following we focus on the system
\begin{equation*}
    F_G^R = R \cdot F_G =
    \begin{bmatrix}
        r_{11} & \cdots & r_{1n} \\
        \vdots & \ddots & \vdots \\
        r_{n1} & \cdots & r_{nn}
    \end{bmatrix}
    \begin{bmatrix}
        F_{G,1} \\
        \vdots \\
        F_{G,n}
    \end{bmatrix}
\end{equation*}
It is easy to verify that for generic choices of the matrix $R$,
there is no complete cancellation of the terms, 
and thus $F_G^R$ is an unmixed system of the form
\begin{equation}\label{equ:kuramoto-rand}
    F_{G,k}^R = c_k^R -
    \sum_{\{i,j\} \in \edges(G)}
        a_{ijk}^R \frac{x_i}{x_j} + 
        a_{jik}^R \frac{x_j}{x_i}
        \quad \text{ for } k = 1,\dots,n
\end{equation}
where $c_k^R$ and $a_{ijk}^R$ are the resulting nonzero coefficients
after collection of similar terms.
This system will be referred to as the \emph{unmixed form}
of the synchronization equations,
and it will be the main focus of the rest of this paper.

\section{Maximum and Generic Root Count}\label{sec:ap}

In this section we briefly review the existing results on the generic root count
of~\eqref{equ:kuramoto-sin}, \eqref{equ:kuramoto-rat}, and~\eqref{equ:kuramoto-rand}.

In~\cite{Baillieul1982}, an upper bound on the number of equilibria of the
Kuramoto model (solutions to~\eqref{equ:kuramoto-sin}) induced by a graph of
$N$ vertices with any coupling strengths is shown to be $\binom{2N-2}{N-1}$.
This bound can be understood as a bi-homogeneous B\'{e}zout number on the
algebraic version~\eqref{equ:kuramoto-rat} or~\eqref{equ:kuramoto-rand}:
Via the map $y_i = x_i^{-1}$, the two systems can be translated into equivalent systems 
that have a bi-degree of $(1,1)$ with respect to the partition 
and are defined in $((x_1,\dots,x_n),(y_1,\dots,y_n))$ 
with the additional conditions that $x_i y_i = 1$ for $i=1,\dots,n$.
It is easy to verify that the bi-homogeneous B\'{e}zout number will be
$\binom{2n}{n} = \binom{2(N-1)}{N-1}$.

Recent studies~\cite{ChenMehtaNiemerg2016} suggest tighter bounds on the number of 
isolated complex solutions may exist when the network is sparsely connected.
When the underlying graph is a cycle,
a sharp bound is established in~\cite{ChenDavisMehta2018} using the
theory of the birationally invariant intersection index as well as a construction
known as the \emph{adjacency polytope} which we shall review briefly here.

A polytope is a bounded intersection of finitely many closed half-spaces.
The adjacency polytope is a polytope that encodes the 
topological information of the Kuramoto network.
Given an undirected graph $G$ with edge set $\mathcal{E}(G)$, its \emph{adjacency polytope} is defined to be
\begin{equation}\label{equ:ap}
    P_{G} = \conv \left\{ 
                \mathbf{e}_i - \mathbf{e}_j 
                \mid 
                \{i,j\} \in \mathcal{E}(G) 
            \right\}
\end{equation}
where we adopt the convention that $\bolde_0 = \bolde_{n+1} = \boldzero$.
That is, the adjacency polytope of $G$ is the convex hull of a set of line segments, each corresponding to an edge in $G$.
$P_G$ is a \emph{lattice polytope} in the sense that all its vertices
have integer coordinates.
Similar constructions have also appeared in other contexts
(e.g. \cite{DelucchiHoessly2016Fundamental,Higashitani2016Interlacing,Matsui2011Roots}).

The \emph{adjacency polytope bound}~\cite{Chen2017Unmixing} 
of a Kuramoto system~\eqref{equ:kuramoto-sin} on the graph $G$ is defined to be 
$n! \operatorname{vol} (P_G)$, the \emph{normalized volume} of $P_G$.
This bound is an upper bound for the number of isolated complex solutions
for the systems~\eqref{equ:kuramoto-rand} and~\eqref{equ:kuramoto-rat}.
Consequently, it is also an upper bound for the number of real solutions that
the original synchronization system~\eqref{equ:kuramoto-sin} has.

In the case of a cycle graph of $N$ nodes, i.e., 
\[G = C_N = (\{0,\dots,N-1\},\{\{0,1\},\dots,\{N-2,N-1\},\{N-1,0\}\}),\]
the recent paper~\cite{ChenDavisMehta2018} established the explicit formula 
$\apbound$ for the adjacency polytope bound.
Furthermore it is shown that this bound coincides with the 
birationally invariant intersection index in $(\C^*)^n$ of the
Kuramoto system~\eqref{equ:kuramoto-rat} as a member of a
family of rational functions.
In this paper, we strengthen this result by producing an explicit
construction of a unimodular triangulation of the adjacency polytope
for cycle graphs and define a homotopy method base on this triangulation.

Before continuing, two remarks are in order.

\begin{remark}\label{rmk:generic}
    The theory of birationally invariant intersection index~\cite{KavehKhovanskii2012Newton,KavehKhovanskii2010Mixed}
    (as well as the general intersection theory~\cite{fulton_introduction_1993} 
    and homotopy continuation theory~\cite{Li2003Numerical,Sommese2005})
    shows that the adjacency polytope bound is ``generically exact''
    in the sense that if one chooses the coefficients of the 
    algebraic Kuramoto equations~\eqref{equ:kuramoto-rat} randomly
    then, with probably one, the total number of isolated complex solutions
    that system has is exactly the adjacency polytope bound $\apbound$.
    Stated more precisely, there exists a nonzero polynomial $D$ 
    whose variables are the coefficients $\{\omega_i\}$ and $\{ a_{ij} \}$ 
    of~\eqref{equ:kuramoto-rand} such that
    for all choices of $\omega_i\}$ and $\{ a_{ij}' \}$ where  $D \neq 0$,
    the total number of isolated complex roots of~\eqref{equ:kuramoto-rat}
    reaches the adjacency polytope bound.
\end{remark}

\begin{remark}
    In specific cases, the adjacency polytope of a graph on $\{0,\dots,N-1\}$ coincides with the type $A_{N-1}$ \emph{root polytope} as defined in \cite{ArdilaEtAlRoots}. This polytope is defined as the convex hull of the generators in $\Z^N$ of the root lattice, generated as a monoid, of the Coxeter group of type $A_{N-1}$.
    Indeed, this is exactly the adjacency polytope $P_G$ where $G$ is the graph for which $0$ is an isolated vertex and the induced subgraph on $\{1,\dots,N-1\}$ is $K_{N-1}$.
    However, in the $C_{N-1}$ and $D_{N-1}$ cases, the constructions of the root polytopes do not coindicde with any adjacency polytopes. 
    
    One should take care when researching root polytopes in the literature, as there are competing notions of root polytopes for root lattices of type $A_{N-1}$.
    One is as we have mentioned, while another considers only \emph{positive} roots of the root lattice (and the origin). 
    This root polytope was introduced and studied in \cite[Section 12]{Postnikov}, with an emphasis on connections to the broad class of polytopes called \emph{generalized permutohedra}.
\end{remark}
\section{A Regular, Unimodular Triangulation of the Adjacency Polytope}\label{sec:triangulation}

A \emph{subdivision} of an $n$-dimensional polytope $P$ is a collection of polytopes $P_1,\dots,P_k \subseteq P$
such that 
\begin{enumerate}
    \item $\dim P_i = n$ for all $i$,
    \item $P_i \cap P_j$ is either empty 
or a face common to both $P_i$ and $P_j$, and
    \item $P = \cup_i P_i$.
\end{enumerate}
A \emph{triangulation}, a.k.a. \emph{simplicial subdivision} 
of a polytope is a subdivision consisting of simplices.
Furthermore, a triangulation is said to be \emph{unimodular}
if all the member simplices are lattice simplices of normalized volume 1.

In order to be used in our homotopy construction, 
the ``regularity'' property of the triangulation is also required.
A triangulation of a polytope is said to be \emph{regular} if
it is the projection of the lower facets of a lifting of the polytope
into one-higher dimension.
Stated more precisely, given a polytope $P = \conv \{ \boldv_1, \dots, \boldv_m \}$
in $\R^n$ and  and weights $\omega_1, \dots, \omega_m \in \R$, the new polytope
\[
    P' = \conv \{ (\boldv_i, \omega_i) \in \R^{n+1} \mid i = 1,\dots,m \}
\]
is a lifting of $P$ into one-higher dimension.
The projections of \emph{lower facets}, that is, the facets whose inner normal vectors have positive last entry,
to the first $n$ coordinates is called a \emph{regular subdivision},
or a \emph{regular triangulation} if all facets are simplices.

For the cycle graph $C_N$ on $N = n+1$ nodes, 
we will construct a unimodular triangulation for the adjacency polytope 
$P_{C_N}$ by finding and triangulating all of its \emph{facets}: the faces of codimension $1$.
Using the set of facets $\mathcal{F}(P_{C_N})$,
a well known subdivision of $P_{C_N}$ can be constructed as the set of pyramids 
formed by the facets and a fixed interior point as the common apex.
That is, fixing any interior point $\mathbf{p} \in P_{C_N}$,
the set
\begin{equation*}
    \{
        \conv F \cup \{\mathbf{p}\} 
        \mid F \in \mathcal{F}(P_{C_N})
    \}
\end{equation*}
forms a subdivision of $P_{C_N}$.
By further triangulating each facet, 
the above subdivision can be refined into a triangulation of $P_{C_N}$.
That is, if $T(F)$ is a triangulation of the facet $F \in \mathcal{F}(P_{C_N})$
then the set
\begin{equation*}
    \{
        \conv C \cup \{\mathbf{p}\} 
        \mid C \in T(F), F \in \mathcal{F}(P_{C_N})
    \}
\end{equation*}
for a fixed interior point $\mathbf{p}$ form a triangulation of $P_{C_N}$.
This is the strategy that we will follow in this section.
The choice of the interior point $\mathbf{p}$ will be the origin $\boldzero$ 
which is an interior point of $P_{C_N}$ since it is the average of
$\bolde_i - \bolde_j$ and $\bolde_j - \bolde_i$ for all edges $\{i,j\}$.

It was shown in \cite{ChenDavisMehta2018} that $P_{C_N}$ is 
unimodularly equivalent to the polytope
\[
    Q_N = \conv\{
        \pm \bolde_1,
        \dots,
        \pm \bolde_n,
        \pm(\bolde_1 + \cdots + \bolde_n)
    \}
\]
via the map $x \mapsto Ax$, 
where $A$ is the $n\times n$ matrix with $1$ on and below the diagonal 
and $0$ everywhere else.
Then \cite[Proposition 12]{ChenDavisMehta2018} and \cite[Remark 4.3]{Nill} identify the facets of $Q_N$.
The geometric structure of this polytope depends on the parity of $N$.
When $N$ is even, the facets can be indexed by the set
\[
    \Lambda_N = \left\{(\lambda_1,\dots,\lambda_N) \in \{-1,1\}^N \mid \sum_{i=1}^N \lambda_i = 0 \right\}
\]
and are of the form 
\[
    F_{\boldlambda} = \conv\{
        \lambda_1 ( -\bolde_1-\bolde_2-\cdots - \bolde_n), 
        \lambda_2 \bolde_1,\dots,\lambda_N \bolde_n, 
        \;\mid\; 
        \boldlambda = (\lambda_1,\dots,\lambda_N) \in \Lambda_N
    \}.
\]
When $N$ is odd, we define $\Lambda_N$ differently: 
in this case, the facets can be indexed by 
$\Lambda_N := \cup_{j=1}^N \Lambda_{j,N}$ where
\[
    \Lambda_{j,N} = \left\{
        (\lambda_1,\dots,\lambda_N) \;\mid\; 
        \lambda_j = 0,\, \lambda_i \in \{-1,1\} 
        \;\text{ for all } i \neq j, \text{ and } 
        \sum_{i=1}^N \lambda_i = 0 \right
    \},
\]
and the facet corresponding to $\boldlambda = (\lambda_1,\dots,\lambda_N) \in \Lambda_{j,N}$
is given by
\[
    F_{\boldlambda} = \conv\{
        \lambda_1(-\bolde_1 - \bolde_2-\cdots - \bolde_n), 
        \lambda_2 \bolde_1,\dots,
        \widehat{\lambda_j \bolde_{j-1}},
        \dots,
        \lambda_N \bolde_n
        \}.
\]
Here, the notation $\widehat{\lambda_j \bolde_{j-1}}$ indicates that element is removed from the list.

From the above constructions, we can see that $Q_N$ is \emph{simplicial}
(i.e., all the facets are simplices) when $N$ is odd, 
but is not simplicial when $N$ is even. 
Via the unimodular equivalence between $Q_N$ and $P_{C_N}$
we have same characterization of the facets of $P_{C_N}$.
As a result of this dichotomy,
the construction of the triangulation in the even and odd $N$ cases
require very different procedures.

\begin{remark}[Unimodular equivalence of facets]\label{rmk:unimod-equiv}
Another important property worth noting is that the facets of $Q_N$ 
are all unimodularly equivalent to each other. 
To see this suppose $F_{\boldlambda}, F_{\boldlambda'}$ are facets of $Q_N$. 
Then, $F_{\boldlambda'} = f(F_{\boldlambda})$ where $f(x) = B_{\boldlambda,\boldlambda'} x$ and 
$B_{\boldlambda,\boldlambda'}$ is the $n \times n$ matrix constructed as follows:
first let $\ell = \lambda_1 \lambda'_1$.
For $1 \leq i \leq n$, note that there is a unique $j$ such that $\lambda_{i+1}$ is the $j^{th}$ instance of $-1$ or the $j^{th}$ instance of $1$ in $(\lambda_2,\dots,\lambda_N)$.
Let $\ell\lambda'_{k+1}$ be the $j^{th}$ instance of $\ell\lambda_{i+1}$ in $(\ell\lambda'_2,\dots,\ell\lambda'_N)$.
Set row $i$ of $B_{\boldlambda,\boldlambda'}$ to be $\ell e_k$.
As a result, $B_{\boldlambda,\boldlambda'}$ is a permutation matrix (up to simultaneous scaling of all entries by $-1$), so $\det B_{\lambda,\lambda'} = \pm 1$, hence $f$ yields a unimodular equivalence.

    Consider, for example, a case with $N=3$ and the choices of
    $\boldlambda = (1,-1,-1,1)$ and $\boldlambda' = (-1,1,-1,1)$.
    In this case, $\ell = \lambda_1 \lambda_1' = -1$.
    Note that $\lambda_2$ is the first instance of $-1$ in $(\lambda_2,\lambda_3,\lambda_4)$.
    Now, $-\lambda_2'$ is the first occurrence of $-\lambda_2$ in $(-\lambda'_2,-\lambda'_3,-\lambda'_4)$.
    So, the first row of $B_{\boldlambda,\boldlambda'}$ is $-\bolde_1$.
    Next, $\lambda_3$ is the second occurrence of $-1$ in $(\lambda_2,\lambda_3,\lambda_4)$, 
    and $-\lambda'_4$ is the second occurrence of $-1$ in $(-\lambda'_2,-\lambda'_3,-\lambda'_4)$, 
    so the second row of $B_{\boldlambda,\boldlambda'}$ is $-\bolde_3$.
    Since $\lambda_3$ is the first occurrence of $1$ in $(\lambda_2,\lambda_3,\lambda_4)$, 
    and $-\lambda'_3$ is the first occurrence of $1$ in $(-\lambda'_2,-\lambda'_3,-\lambda'_4)$, 
    we have that the third row of $B_{\lambda,\lambda'}$ is $-\bolde_2$:
    \[
        B_{\boldlambda,\boldlambda'} =
        \begin{bmatrix}
            -1 & 0 & 0 \\
            0 & 0 & -1 \\ 
            0 & -1 & 0 
        \end{bmatrix}.
    \]
\end{remark}

\begin{remark}[Point configuration]\label{rmk:pt-conf}
    In order to be used in a homotopy construction,
    a stronger triangulation is needed.
    Define the point set
    \begin{equation*}
        S_{C_N} = \{ \boldzero \} \cup  \{ \bolde_i - \bolde_j \mid \{i,j\} \in \mathcal{E}(C_N) \}
    \end{equation*}
    This set is known as the \emph{support} of the unmixed system~\eqref{equ:kuramoto-rand}
    as it collects the exponents (as points) of all the terms appearing in that system.
    It is easy to see that $P_{C_N} = \conv S_{C_N}$
    since $\boldzero$ is an interior point of $P_{C_N}$
    (as $\boldzero = \frac{1}{2} (\bolde_i - \bolde_j) + \frac{1}{2} (\bolde_j - \bolde_i)$).
    In our constructions, we will require all simplices in a triangulation
    to have vertices within the set $S_{C_N}$.
    This is known as a triangulation of a \emph{point configuration}.
\end{remark}

In the rest of this section, we describe the construction of 
regular unimodular triangulation of $P_{C_N}$ in the cases with
even and odd $N$ respectively.

\subsection{Even $N$}

For the entirety of this subsection, we assume that $N$ is even.
From the preceding discussion, 
we know that all of the facets of $P_{C_N}$ are unimodularly equivalent 
due to transitivity of equivalence relations.
In particular, all facets of $P_{C_N}$ are unimodularly equivalent to
\[
    \conv \{
        \bolde_0-\bolde_1,
        \bolde_1-\bolde_2,\dots,
        \bolde_{\lfloor \frac{n}{2} \rfloor}-\bolde_{\lfloor \frac{n}{2} \rfloor+1}, 
        -(\bolde_{\lfloor \frac{n}{2} \rfloor+1}-\bolde_{\lfloor \frac{n}{2} \rfloor+2}),\dots,-(\bolde_{n-1}-\bolde_n),-\bolde_n
    \}.
\]
Let $G_{\boldlambda}$ denote the facet of $P_{C_N}$ obtained by applying 
$A^{-1}$ to all points in $F_{\boldlambda}$.
It will be important to keep in mind that 
$\pm A^{-1}(\bolde_1 + \cdots + \bolde_n) = \pm \bolde_1$.
We can then produce a subdivision of $P_{C_N}$ by setting
\[
    G^{\mathbf{0}}_{\boldlambda} = \conv\{\mathbf{0},G_{\boldlambda}\}.
\]
and ranging over all $\boldlambda \in \Lambda_N$.
    
To aid us in what follows, we establish the following lemma.
Recall that in $\R^n$, we use the convention $\mathbf{e}_0 = \mathbf{e}_{n+1} = \mathbf{0}$.

\begin{lemma}\label{lem: two triangulations}
    Let $V_N = \{\boldv_0,\dots,\boldv_N\}$ denote the vertices of 
    $G^{\mathbf{0}}_{\boldlambda}$ such that
    \[
        \boldv_i = 
        \begin{cases}
            \boldzero & \text{ if } i = 0,\\
           \lambda_i(\bolde_{i-1} - \bolde_i) & \text{ if } 1 \leq i \leq N,
        \end{cases}
    \]
    When $N$ is even, each $G^{\mathbf{0}}_{\boldlambda}$ has exactly two triangulations:
    \[
        \Delta_+(G^{\mathbf{0}}_{\boldlambda}) = 
        \{\conv\{V_N \setminus \{\boldv_i\}\} \mid \lambda_i = \lambda_1 \}
    \]
    and
    \[
        \Delta_-(G^{\mathbf{0}}_{\boldlambda}) = 
        \{\conv\{V_N \setminus \{\boldv_i\}\} \mid \lambda_i = -\lambda_1 \}
    \]
    Moreover, both of these triangulations are regular.
\end{lemma}

\begin{proof}
    Let $N = 2k$.
    Note that for each $\boldlambda \in \Lambda_N$, 
    $\dim G^{\boldzero}_{\boldlambda} = n$ and 
    $G^{\mathbf{0}}_{\boldlambda}$ has $n+2$ vertices.
    Thus, there is a unique (up to simultaneous scaling of the coefficients) 
    affine dependence of the form 
    \[
        \sum_{i=0}^{N} c_i \boldv_i = \boldzero
    \]
    satisfying $\sum c_j = 0$ with $c_0,\dots,c_{n+1} \in \R$.
    Without loss of generality, 
    we may choose $c_0 = 0$ and $c_i = \lambda_ik/N$ for $i > 0$.
    
    The desired conclusions for the lemma then follow from \cite[Lemma~2.4.2]{TriangulationsBook}.
    Specifically,
    \[
        \Delta_+(G^{\mathbf{0}}_{\boldlambda}) = 
        \left\{ \conv\{ V \setminus \{\boldv_i\}\} \mid \lambda_i = \lambda_1 \right\}
    \]
    is the triangulation of $G^{\mathbf{0}}_{\boldlambda}$ corresponding to 
    the height vector $(\omega_0,\dots,\omega_N)$ where
    \[
        \omega_i = 
        \begin{cases}
            0 & \text{ if } c_i \leq 0, \\
            1 & \text{ if } c_i > 0
        \end{cases}
    \]
    and
    \[
        \Delta_-(G^{\mathbf{0}}_{\boldlambda}) = 
        \left\{ \conv\{ V \setminus \{\boldv_i\}\} \mid \lambda_i = -\lambda_1 \right\}
    \]
    is the triangulation corresponding to the heights
    \[
        \omega_i = 
        \begin{cases}
            0 & \text{ if } c_i \geq 0, \\
            1 & \text{ if } c_i < 0.
        \end{cases}
    \]
\end{proof}

We will be concerned with the particular lifting function 
$\omega : S_{C_N} \to \Z$ given by
\[
    \omega(\mathbf{a}) = 
    \begin{cases}
        0 & \text{ if } \mathbf{a} = \boldzero \\
        2 & \text{ if } \mathbf{a} = \pm \bolde_1 \\
        1 & \text{ otherwise } 
    \end{cases}
\]
as this will induce the desired regular, unimodular triangulation 
$\Delta_N$ of $P_{C_N}$.
To help us with notation, we will define
\[
    \Omega_{\omega}(P) = \conv \{ \omega(v) \mid v \in P \cap S_{C_N} \}
\]
for any polytope $P$ whose vertices are a subset of $S_{C_N}$.
If $X$ is a collection of polytopes whose vertices are subsets of $S_{C_N}$, then we let $\Omega_{\omega}(X) = \{\Omega_{\omega}(P) \mid P \in X\}$.

First, we identify normal vectors of simplices in $\Omega_{\omega}(\Delta_+(G^{\mathbf{0}}_{\lambda}))$.
Recall that a vector is \emph{upward-pointing} if its final coordinate is positive.

\begin{lemma}\label{lem:normal-even}
    Let $N$ be even. 
    If $\boldlambda \in \Lambda_N$ and $\lambda_1 = 1$, 
    then the upward-pointing inner normal vectors for all simplices in
    $\Omega_{\omega}(\Delta_+(G^{\mathbf{0}}_{\boldlambda}))$ are
    $x_{\boldlambda} = (x_1,\dots,x_{n+1})$ where 
    \[
        x_i = 
        \begin{cases}
            \sum_{j=1}^i \lambda_j & \text{ if } 1 \leq i < n+1, \\
            1 & \text{ if } i = n+1
        \end{cases}
    \]
    as well as 
    \[
        y_{\boldlambda,j} = x + \sum_{k=1}^{j-1} \bolde_k 
    \]
    for each $j>1$ such that $\lambda_j = 1$.
    If $\lambda_1 = -1$, then the upward-pointing inner normal vectors are 
    $x_{\lambda} = \sigma(x_{-\boldlambda}),y_{\boldlambda,j} = \sigma(y_{-\boldlambda,j})$ 
    where $\lambda_j > 0$ and where $\sigma$ is the map that negates the first $n$ coordinates.
\end{lemma}

\begin{proof}
    First observe that, by construction, 
    each vector under consideration is upward-pointing.
    Next, let $\lambda_1 = 1$.
    It is then straightforward to verify that
    \[
        \langle \boldx, \lambda_j (\bolde_{j-1}-\bolde_j) + \bolde_{n+1} \rangle = -\lambda_j^2 + 1 = 0
    \]
    for all $1 < j \leq n+1$.
    Following this same process, one may verify that 
    the hyperplane for which $y_{\lambda,j}$ is normal contains all vertices 
    of $G^{\mathbf{0}}_{\boldlambda}$ except $\lambda_j(\bolde_{j-1} - \bolde_j)$. 
    
    Finally notice that if $\boldlambda' = -\boldlambda$, 
    then $-T \in \Delta_+(G^{\mathbf{0}}_{\boldlambda'})$ 
    for each cell $T \in \Delta_+(G^{\mathbf{0}}_{\boldlambda})$.
    It directly follows that $\sigma(x_{\boldlambda})$ and $\sigma(y_{\boldlambda,j})$ 
    are the upward-pointing inner normal vectors of the simplices in $\Delta_+(G_{\boldlambda}^{\boldzero})$ for all $\boldlambda$ satisfying $\lambda_1 = -1$.
\end{proof}

\begin{theorem}
    Let $N$ be even. 
    The height function $\omega : S_G \to \Z$ given by
    \begin{equation}\label{equ:weights-even}
        \omega(\mathbf{a}) = 
        \begin{cases}
            0 & \text{ if } \mathbf{a} = \mathbf{0}, \\
            2 & \text{ if } \mathbf{a} = \pm e_1, \\
            1 & \text{ otherwise } 
        \end{cases}
    \end{equation}
    induces a regular unimodular triangulation $\Delta_N$ 
    of the point configuration $S_{C_N}$.
    Specifically,
    \begin{equation}\label{equ:triangulation-even}
        \Delta_N = \bigcup_{\lambda \in \Lambda_N}\Delta_+(G^{\mathbf{0}}_{\lambda})
    \end{equation}
\end{theorem}

\begin{proof}
    Let $\Delta_N$ denote the regular subdivision of $P_{C_N}$ induced by $\omega$.
    For $\lambda \in \Lambda_N$, consider the vectors $x_{\lambda}, y_{\lambda,j}, \sigma(x_{\lambda}), \sigma(y_{\lambda,j})$, as defined in Lemma~\ref{lem:normal-even}.
    First, we focus on $x_{\lambda}$.
    We have already seen that each vertex of $G^{\mathbf{0}}_{\lambda}$ except for $-\lambda_1e_1$ lies on the hyperplane with normal vector $x$.
    In fact, it is straightforward to check that 
    \[
        \langle x,-\lambda_j(\bolde_{j-1} - \bolde_j) + \bolde_{n+1} = \lambda_j^2 + 1 > 0
    \]
    for all $1 < j \leq n+1$, and that
    \[
        \langle x,\pm \bolde_1 + 2\bolde_{n+1}\rangle = \pm 1 + 2 > 0,
    \]
    so $x$ defines a facet of $\Omega_{\omega}(P_{C_N})$.

    Following this same process, one may verify that $y_{\lambda,j}$ 
    defines a facet containing all vertices of $G^{\mathbf{0}}_{\lambda}$ 
    except $\lambda_j(\bolde_{j-1} - \bolde_j)$.
    By an argument that is symmetric in the first $n$ coordinates, $\sigma(x_{\lambda})$ and $\sigma(y_{\lambda,j})$ also defined facets of $\Omega_{\omega}(P_{C_N})$.
    
    Ranging over all $\lambda \in \Lambda_N$, we have identified a collection of simplices $C$ that are lower facets of $\Omega_{\omega}(P_{C_N})$.
    Projecting each $C$ back down to $\R^n$, we get
    \begin{equation}
        \bigcup_{\lambda \in \Lambda_N}\Delta_+(G^{\mathbf{0}}_{\lambda}) \subseteq \Delta_N.
    \end{equation}
    
    In fact, this set covers $P_{C_N}$ completely: let $a \in P_{C_N}$.
    Then for some nonzero $c \in \R$, $ca$ is on the boundary of $P_{C_N}$.
    Thus, $ca \in G_{\boldlambda}$ for some $\lambda \in \Lambda_N$, and $a \in G^{\mathbf{0}}_{\lambda}$.
    Therefore, $a \in C$ for some cell $C \in \Delta_+(G^{\mathbf{0}}_{\lambda})$.
    
    Together, this shows that $\Delta_N$ is a triangulations of $P_{C_N}$, and is the regular triangulation induced by $\omega$.
    To see that this triangulation is unimodular, 
    recall that all simplices in $\Delta_N$ are unimodularly equivalent to the simplex whose nonzero vertices are 
    \[
        \bolde_0-\bolde_1,\bolde_1-\bolde_2,\dots,\bolde_{\lfloor \frac{n}{2} \rfloor}-\bolde_{\lfloor \frac{n}{2} \rfloor+1}, -(\bolde_{\lfloor \frac{n}{2} \rfloor+1}-\bolde_{\lfloor \frac{n}{2} \rfloor+2}),\dots,-(\bolde_{n-1}-\bolde_n),-\bolde_n.
    \]
    Placing these vertices as the columns of a matrix, in this order, results in a lower-triangular matrix with determinant $\pm 1$.
    Thus, the corresponding simplex, and therefore all simplices in $\Delta_N$, are unimodular.
    This completes the proof.
\end{proof}

\begin{remark}\label{rmk:no-01}
    The direct acyclic decomposition scheme developed in~\cite{Chen2019Directed}
    is equivalent to the process of computing a regular subdivision of the adjacency polytope
    induced by certain $0/1$ weights.
    It is shown that for certain graphs, this process will produced regular unimodular triangulations
    which is desired due to their connection to \emph{primitive decomposition} of a Kuramoto network.
    Here, however, we can see this is not possible in general.
    In particular, 
    with the aid of Macaulay2~\cite{M2} to test all $2^9 = 512$ possible $0/1$ weight orders for $P_{C_4}$, 
    we verified that only $4$ choices of weights produce a triangulation of the polytope, 
    and of these, none are unimodular.
    So, in then sense of bounding the heights of the lattice points of $\omega(P_{C_N})$ 
    for all even $N$, using only nonnegative integer heights, to produce a 
    regular unimodular triangulation, the $\omega$ given in this subsection is best possible.
\end{remark}


\subsection{Odd $N$}

For the entirety of this subsection, we assume that $N$ is odd.
Recall that in this case, the facets of $Q_N$ consist of all sets of the form 
\[
    F_{j,\lambda} = \conv\{\lambda_1(-\bolde_1-\bolde_2-\cdots - \bolde_n),
    \lambda_2 \bolde_1,\dots,\widehat{\lambda_j \bolde_{j+1}},\dots,\lambda_N \bolde_n\}
\]
Tracing this back to $P_{C_N}$, we find that its facets are of the form
\[
    G_{\boldlambda} = \conv\{
        \lambda_1(\bolde_0-\bolde_1), \dots, 
        \widehat{\lambda_j(\bolde_{j-1}-\bolde_j)},\dots,\lambda_N(\bolde_n-\bolde_N) \mid 
        \lambda \in \Lambda_{j,N}
    \}
\]
Set
\[
    G^{\mathbf{0}}_{\lambda} = \conv\{
        \lambda_i(\bolde_i-\bolde_{i+1}) \mid 
        (\lambda_1,\dots,\lambda_N) \in \Lambda_{j,N}
    \},
\]
and let
\begin{equation}\label{equ:triangulation-odd}
    \Delta_N = \{G^{\mathbf{0}}_{\lambda} \mid \lambda \in \Lambda_N\}.
\end{equation}
By construction, since each $G_{\boldlambda}$ is a simplex, 
$\Delta_N$ is a triangulation of $P_{C_N}$.
It is straightforward to check that the matrix whose columns are the nonzero vertices of $G^{\mathbf{0}}_{\lambda}$ has determinant $\pm 1$ for each $\lambda \in \Lambda_N$, so $\Delta_N$ is a unimodular triangulation.

Now, let $\omega : S_{C_N} \to \Z$ be the height function given by
\begin{equation}\label{equ:weights-odd}
    \omega(\mathbf{a}) = 
    \begin{cases}
        0 & \text{ if } \mathbf{a} = \boldzero \\
        1 & \text{ otherwise}.
    \end{cases}
\end{equation}
It is clear from this choice that the lower facets of the lifted polytope 
$\Omega_{\omega}(P_{C_N})$ are of the form
\[
    \conv\{\mathbf{0}, G_{\boldlambda} \times \{1\}\},
\]
so their projections back onto $\R^n$ are exactly 
the simplices $G^{\mathbf{0}}_{\lambda}$ for all $\lambda \in \Lambda_N$.
With this work, we have shown the following.

\begin{proposition}
    The set $\Delta_N$ is a regular, unimodular triangulation of $P_{C_N}$, 
    and is induced by the height function $\omega$ in \eqref{equ:weights-odd}.
\end{proposition}

We can, in fact, be more specific when identifying the lower facets of 
$\Omega_{\omega}(P_{C_N})$.

\begin{corollary}\label{cor:normal-odd}
    The upward-pointing inner normal vectors for 
    $\Omega_{\omega}(G^{\mathbf{0}}_{\lambda})$ are $x=(x_1,\dots,x_N)$ where
    \[
        x_k = \begin{cases}
            \sum_{i=1}^k \lambda_i & \text{ if } i < N, \\
            1 & \text{ if } i = N
        \end{cases}
    \]
    for all $\lambda \in \Lambda_N$.
\end{corollary}

\begin{proof}
    Let $\Omega_{\omega}(G^{\boldzero}_{\boldlambda})$ be a lower facet of 
    $\Omega_{\omega}(P_{C_N})$ for some $\boldlambda \in \Lambda_{j,N}$, 
    and select a nonzero vertex $\boldv$ of the facet.
    Since this vertex is nonzero, we know $\boldv$ is of the form 
    $\boldv = \lambda_{r+1}(\bolde_r - \bolde_{r+1}+\bolde_N$ for some $r \neq j$. 
    Then 
    \[
        \langle \boldx, \boldv\rangle = 
        \lambda_{r+1} \left(\sum_{i=1}^r \lambda_i - \sum_{l = 1}^{r+1} \lambda_l\right) + 1 =
        -\lambda_{r+1}^2 + 1 = 0.
    \]
    Thus, $\boldx$ is normal to $\Omega_{\omega}(G^{\boldzero}_{\boldlambda})$.
\end{proof}

\section{Cell enumeration algorithm}\label{sec:algo}

In this section, we briefly summarize the algorithm for
for constructing a regular unimodular triangulation for the
adjacency polytope $P_{C_N}$ as proposed above.
Here, we shall focus only on the enumeration of all the
upward pointing inner normal vectors of the lifted polytope
$\Omega(P_{C_N})$ of the point configuration $S_{C_N}$,
since these are directed used in the homotopy construction
to be described in~\cref{sec:homotopy}.
Moreover, these objects directly correspond to tropical stable intersections
as we shall discuss in detail in~\cref{sec:tropical}.
Once a normal vector $\boldv$ is obtained, the vertices of the corresponding cell
can be found easily by computing the minimizing set of the linear functional
$\langle \cdot \,,\, \boldv \rangle$.

The algorithm \textsf{EnumerateNormals}($N$) for enumerating inner normal vectors 
is listed in~\cref{alg:cells}.
It takes the argument $N$,
which is the number of nodes in the cycle graph and produces
the upward pointing inner normal vectors of the
lower facets of $\Omega(P_{C_N})$ which are in one-to-one correspondence
with the simplices in the regular unimodular triangulation $\Delta_N$.

\begin{algorithm}[h]
    \caption{\textsf{EnumerateNormals($N$)}: Enumeration of upward pointing inner normals}\label{alg:cells}
    \renewcommand{\algorithmicrequire}{\textbf{Input:}}
    \renewcommand{\algorithmicensure}{\textbf{Output:}}
    \begin{algorithmic}
        \REQUIRE $N \in \Z^+, N > 2$.
        \ENSURE Set $C$ of all upward pointing inner normals.
        \STATE $C \gets \varnothing$
        \FORALL{$(\lambda_1,\dots,\lambda_n) \in \Lambda_N$}
            \FOR{$k = 1,\dots,n$}
                \STATE $x_k \gets \sum_{i=1}^k \lambda_i$
            \ENDFOR
            \STATE $x_N \gets 1$
            \STATE $\boldx \gets [ x_1, \dots, x_N ]$
            \STATE $C \gets C \cup \{ \boldx \}$
            \IF{$N$ is even and $N > 2$}
                \FOR{$j=1,\dots,n$}
                    \IF{$\lambda_j = 1$}
                        \STATE $\boldy \gets \boldx + \sum_{k=1}^{j-1} \bolde_k$
                        \STATE $C \gets C \cup \{ \boldy \}$
                    \ENDIF
                \ENDFOR
            \ENDIF
        \ENDFOR
        \RETURN $C$
    \end{algorithmic}
\end{algorithm}

Note that this algorithm is \emph{pleasantly parallel}
since the description of vectors associated with indices $\boldlambda \in \Lambda_N$
are independent from one another.
The cost for producing each normal vector is $O(N^2)$,
and no additional storage is needed.
\section{Software implementation}\label{sec:implementation}

The main algorithms for generating the cells in the regular unimodular triangulation
$\Delta_N$ of $P_{C_N}$ is implemented in an open source Python package 
\texttt{kap-cycle}~\cite{Chen2018KAPCycle}. 
In addition to the cells, this package also produces the upward pointing inner normals
corresponding to each cell.
That is, it provides all the necessary information for bootstrapping the 
adjacency polytope homotopy proposed in~\cref{sec:homotopy}.


\section{The Adjacency Polytope Homotopy for Kuramoto Equations}\label{sec:homotopy}

We now return to the problem of find all isolated complex solutions of~\eqref{equ:kuramoto-rat}.
Equivalently, these are the solutions of $F_{C_N}^R$ defined in~\eqref{equ:kuramoto-rand}.
Utilizing the unimodular regular triangulation of the adjacency polytope $P_{C_N}$,
in this section we describe a specialized \emph{polyhedral homotopy}
\cite{huber_polyhedral_1995}
construction for locating all of these complex solutions
\emph{yet avoid the computationally expensive steps associated with polyhedral homotopy}.

Consider the function $H_{C_N} : \C^n \times \C \to \C^n$ with 
$H_{C_N}(\mathbf{x},t) = (H_{{C_N},1},\dots,H_{{C_N},n})$ given by
\begin{equation}
    H_{C_N,k}(x_1,\dots,x_n,t) = 
    c_{k}^R - 
    \sum_{\{i,j\} \in \edges(C_N)} 
    \left(
        a_{ijk}^R
        \frac{x_i}{x_j} + 
        a_{jik}^R \frac{x_j}{x_i}
    \right) 
    t^{\omega_{ij}}
    \quad \text{ for } k = 1,\dots,n
    \label{equ:homotopy}
\end{equation}
where $\omega_{ij} = \omega(\bolde_i - \bolde_j)$ 
as given in~\eqref{equ:weights-even} or~\eqref{equ:weights-odd}
depending on the parity of $N$.
Clearly, $H_{C_N}(\boldx,1) = F_{C_N}^R(\boldx)$.
As $t$ varies strictly between 0 and $1$ within the interval $[0,1]$, 
$H_{C_N}(\mathbf{x},t)$ represents a smooth deformation of the system $F_{C_N}^R$.
We shall show that under this deformation, the corresponding complex roots also vary smoothly.
Thus, the deformation forms smooth paths reaching the complex roots of $F_{C_N}^R$
and, equivalently, that of the algebraic synchronization system $F_{C_N}$.

\begin{proposition}
    For generic choices of the parameters, 
    the solution set of $H_{C_N}(\boldx,t) = \boldzero$ within $\C^n \times (0,1)$
    consists of $\apbound$ smooth curves 
    that are smoothly parametrized by $t \in (0,1)$,
    and the limit point of these curves as $t \to 1$ are precisely the 
    isolated complex solutions of $F_{C_N}(\boldx) = \boldzero$.
\end{proposition}

This is a special version of the smoothness condition 
for the polyhedral homotopy,
and its proof can be found in~\cite{huber_polyhedral_1995,Li2003Numerical}.
Here we include a variation of the proof adapted from~\cite{Li2003Numerical} for completeness,
as the special choice of the lifting function prevents
us from directly applying the general theorems which require generic liftings.

\begin{proof}
    As proved in~\cite{ChenDavisMehta2018}, for generic choices of the parameters
    the system $H_{C_N}(\boldx,1) \equiv F_{C_N}^R$ is in general position
    with respect to the adjacency polytope bound,
    i.e., it has the maximum number of isolated complex solutions.
    
    Also note that for any $t \ne 0$, 
    $H(\boldx,t)$ is has the same form as~\eqref{equ:kuramoto-rand}
    since the effect of $t$ is only in the scaling the coefficients.
    We shall show that $H(\boldx,t)$ remains a generic member 
    of~\eqref{equ:kuramoto-rat} for all $t \in (0,1]$ and hence
    the all complex solutions of $H(\boldx,t) = \boldzero$ 
    (as a system in $\boldx$ only) are isolated and the total number
    matches the adjacency polytope bound $\apbound$.
    
    As noted in~\cref{rmk:generic}, the genericity condition is characterized by
    an algebraic function $D$, the discriminant, which is a polynomial in the coefficients
    $c_k^R$ and $a_{ijk}^R t^{\omega_{ij}}$ for $k=1,\dots,n$ and  $\{i,j\} \in \edges(C_N)$
    such that $F(\boldx) = H_{C_N}(\boldx,t)$ is generic with respect to the 
    adjacent polytope bound precisely when $D \ne 0$.
    Consider the univariate polynomial 
    \[
        g(t) = D((c_i^R)_{i=1}^n,(a_{ijk}^R t^{\omega_{ij}})_{\{i,j\} \in \edges(C_N),k=1,\dots,n}).
    \]
    By our genericity assumption, $g(1) \neq 0$,
    and therefore the polynomial $g(t)$ is not the zero polynomial.
    It then has finitely many zeros within the unit disk of $\C$, say, 
    $r_1 e^{\tau_1},\dots,r_\ell e^{\tau_\ell}$ for some $\ell \in \Z^+$.
    Picking a real value $\tau \in [0,2\pi]$ such that 
    $\tau \neq \tau_k$ for $k=1,\dots,\ell$ will ensure
    $g( e^\tau t ) \neq 0$ for all $t \in (0,1)$.
    But $g( e^\tau t)$ describes the discriminant condition for the system
    \begin{equation*}
        H_{{C_N},k}(x_1,\dots,x_n, e^\tau t) = 
        c_k^R - 
        \sum_{\{i,j\} \in \mathcal{E}(C_N)} 
        \left(
            a_{ijk}^R \frac{x_i}{x_j} + 
            a_{jik}^R \frac{x_j}{x_i}
        \right) 
        e^{\omega_{ij} \tau} t^{\omega_{ij}},
        \quad\text{for } k=1,\dots,n
    \end{equation*}
    which implies $H_{C_N}(\boldx, e^\tau t)$ is in general position
    for all $t \in (0,1)$.
    
    Notice that the map $\tau \mapsto e^{\omega_{ij} \tau}$ is finite-to-one,
    and the map $a_{ijk}^R \mapsto a_{ijk}^R \cdot e^{\omega_{ij} \tau}$
    is a nonsingular linear transformations on the coefficients $\{ a_{ijk}^R \}$,
    which preserves genericity.
    We can conclude that for generic choices of the coefficients,
    $H_{C_N}(\boldx,t)$ will be in general position for $t \in (0,1)$.
    
    This shows that at any fixed $t \in (0,1)$,
    all solutions of $H(\boldx,t) = \boldzero$ in $\C^n$ are isolated
    and the total number is exactly the adjacency polytope bound $\apbound$.
    A direct application of the homotopy continuation theory
    is then sufficient to establish that the solution set of $H(\boldx,t) = \boldzero$
    in $\C^n \times (0,1)$ forms paths that are smoothly parametrized by $t$.
    Furthermore, by continuity, the limit points of these paths as $t \to 1$
    must be all the solutions of $F_{C_N}^R(\boldx) = \boldzero$
    which is identical to that of $F_{C_N}(\boldx) = \boldzero$.
\end{proof}

The equation $H_{C_N}(\boldx,t) = \boldzero$ defines finitely many
smooth paths in $\C^n \times (0,1)$ reaching all of the isolated complex solutions
of the target synchronization system $F_{C_N}(\boldx) = \boldzero$.
The starting point of these paths at $t=0$, however, 
cannot be determined directly since $H_{C_N}(\boldx,0) = (c_1,\dots,c_n)$
which has no root in $\C^n$.
This obstacle is surmounted via a technique similar to the main construction 
in polyhedral homotopy~\cite{huber_polyhedral_1995}.

Recall that $\Delta_N$ is the set of cells forming the
unimodular triangulation of the adjacency polytope $P_{C_N}$
(defined in~\eqref{equ:triangulation-even} or~\eqref{equ:triangulation-even}
depending on the parity of $N$).
For each cell $T \in \Delta_N$, we define the subset of (directed) edges
\[
    \edges(T) = \{ (i,j) \in \edges(G) \mid \bolde_i - \bolde_j \in T \}.
\]
Here, we do not assume the symmetry of edges, i.e., $(i,j) \in \edges(T)$
does not imply $(j,i) \in \edges(T)$.
Define the \emph{cell system} $F_T = (F_{T,1},\dots,F_{T,n})$
associated with the cell $T \in \Delta_N$ given by
\begin{equation}\label{equ:cell}
    F_{T,k} (\mathbf{x}) =
    c^R_k -
    \sum_{(i,j) \in \mathcal{E}(T)} a^R_{ijk} \, \frac{x_i}{x_j}
    \quad \text{ for } k=1,\dots,n.
\end{equation}
This system can be considered as a subsystem of the 
unmixed synchronization system~\eqref{equ:kuramoto-rand}
in the sense that it involves a subset of the terms in that system:
only those terms corresponding to points in $T$.
Indeed, $T$ is exactly the Newton polytope of the corresponding cell system.

\begin{remark}\label{rmk:cell-system}
    The cell systems defined here are refinements of the \emph{facet systems}
    studied in~\cite{Chen2019Directed}.
    Indeed, for odd $N$, they are exactly the facet systems since each $T \in \Delta_N$
    is the convex hull of a facet of the adjacency polytope $P_{C_N}$ together with the origin.
    For even $N$ values, however, the cell systems will be significant refinement
    of facet systems defined in~\cite{Chen2019Directed}.
    This distinction will be explained in detail in~\cref{sec:dag}.
\end{remark}

Here, each cell $T \in \Delta_N$ is a full-dimensional lattice simplex 
with normalized volume 1 (a primitive simplex).
From classical theory from toric algebraic geometry, we can deduce that
the corresponding cell system has a unique solution.

\begin{lemma}
    For generic choices of $\{c_i\}_{i=1}^n$ and $\{a_{ij}'\}_{\{i,j\} \in \mathcal{E}(C_N)}$,
    each system of Laurent polynomial equations $F_{T}(\boldx) = \boldzero$ for $T \in \Delta_N$
    has a unique complex solution, and this solution is isolated and nonsingular.
\end{lemma}

A direct algebraic proof is also given in~\cite[Theorem 4]{Chen2019Directed},
as the system $F_T$ corresponds to a \emph{primitive} subnetwork
which has a unique generalized frequency synchronization configuration.
More importantly, as shown in that proof (\cite[Remark 5]{Chen2019Directed},
after Gaussian elimination,
such a system can be reduced to a special form of binomial system 
which can be solved in linear time, i.e., 
in $O(n)$ complex multiplications and divisions 
with no additional memory needed.

With this result, we modify the homotopy~\eqref{equ:homotopy}
so that it defines a solution path that starts from the unique solution of 
the cell system $F_T$.
This is essentially a specialized polyhedral homotopy~\cite{huber_polyhedral_1995} 
of the ``unmixed'' form using the triangulation found in the previous section.

\begin{definition}[Adjacency polytope homotopy]\label{def:ap-homotopy}
    For each cell $T \in \Delta_N$, 
    let $(\alpha_1,\dots,\alpha_n,1) \in \R^{n+1}$
    be the associated upward pointing inner normal vector as given
    in~\cref{lem:normal-even} and~\cref{cor:normal-odd}.
    We define the \emph{adjacency polytope homotopy} induced by this cell
    as the function $H_T = (H_{T,1}, \ldots, H_{T,n}) : \C^n \times [0,1]$
    given by 
    \begin{equation}
        H_{T,k}(\mathbf{y},t) = 
        H_{C_N,k}(
            y_1 \, t^{\alpha_1} \,,
            \, \ldots, \,
            y_n \, t^{\alpha_n}
        ).
    \end{equation}
\end{definition}

Recall that the collection of cells in $\Delta_N$ form a 
regular triangulation of the adjacency polytope $P_{C_N}$ 
which is also the Newton polytope of~\eqref{equ:kuramoto-rand}.
By applying the construction of unmixed form of polyhedral homotopy~\cite{huber_polyhedral_1995}, 
we obtain the desired result:
A homotopy that can locate all isolated complex solution of the
algebraic synchronization equation~\eqref{equ:kuramoto-rat}.

\begin{theorem}
    For generic choices of $\{c_i\}_{i=1}^n$ and  $\{a_{ij}'\}_{\{i,j\} \in \mathcal{E}(C_N)}$ 
    \begin{enumerate}
        \item 
            The solution set of $H_{T} = \boldzero$ within $\C^n \times (0,1)$
            consists of a finite number of smooth paths parametrized by $t$,
            and the limit point of these paths as $t \to 1$ are precisely the isolated solutions of
            $F_{C_N}(\boldx) = \boldzero$ in $\C^n$.
        \item
            Among them, there is a unique path $C_T(t)$ whose limit point $C_T(0) = \lim_{t \to 0^+} C_T(t)$ 
            is the unique solution of $F_{T}(\boldx) = \boldzero$.
        \item
            The set of end points $\{ C_T(1) \mid T \in \Delta_N \}$
            of paths induced by all cells is exactly the isolated $\C$-solution set of
            $F_{C_N}(\boldx) = \boldzero$.
    \end{enumerate}
\end{theorem}

Compared to an direct application of the unmixed form of the polyhedral homotopy,
the above construction has great computational advantages as summarized in~\cref{tab:benefits}.

\begin{remark}\label{rmk:numerical}
    From the viewpoint of numerical analysis, the stability of a homotopy formulation is a
    deep and complex problem that is outside the scope of this paper.
    Here we only comment one distinct advantage of the adjacency polytope homotopy
    over a direct application of polyhedral homotopy method.
    In practical implementations of polyhedral homotopy, 
    it is well known that the distribution of the exponents of the $t$ parameter in
    the homotopy plays a crucial role in the numerical stability of the 
    homotopy algorithm~\cite{GaoLiVerscheldeWu2000Balancing}.
    In particular, if the exponents of $t$ spread over a wide range, the problem of
    tracking the homotopy paths can become extremely ill-conditioned 
    and standard algorithms for path tracking become unstable.
    While many techniques have been developed to deal with this issue,
    it is much preferred if this problem can be avoided in the first place.
    In our construction, the exponents of $t$ in both~\eqref{equ:homotopy}
    and~\cref{def:ap-homotopy} involve small integers $\{0,1,2\}$,
    this ensures that the exponents of $t$ in $H_T(\boldy,t)$ consist of only small 
    positive integers for relatively small $N$ values.
\end{remark}


\begin{table}[ht]
    \centering
    \begin{tabular}{ccc}
        \toprule
        & 
        \makecell{Direct application of\\{}polyhedral homotopy} & 
        \makecell{Adjacency polytope\\{}homotopy} \\ \midrule
        \makecell{Construction of\\{}the homotopy} & 
        \makecell{Requires the costly step of\\{}mixed cells computation} &
        \makecell{Each homotopy $H_T$\\{}only requires one cell\\{}generated by~\cref{alg:cells}} \\ \midrule
        \makecell{Starting systems} &
        \makecell{Binomial systems which are\\{}usually solved in $O(n^3)$ time\\{}and additional memory} &
        \makecell{Special ``primitive'' system\\{}that can be solved in $O(n)$\\{}time and requires no memory}\\ \midrule
        \makecell{Lifting function} &
        \makecell{Use random image and requires\\{}numerical conditioning techniques} &
        \makecell{Use values $\{0,1,2\}$ and will\\{}not directly cause instability} \\
        \bottomrule\\
    \end{tabular}
    \caption{Computational advantages of the adjacency polytope homotopy}
    \label{tab:benefits}
\end{table}



\section{Interpretations}\label{sec:interpretation}

In this section we interpret our main results in a wider context
and draw connections to closely related problems.
Even though the main goal is to construct an efficient homotopy method for 
locating complex synchronization configurations for Kuramoto networks supported on cycle graphs,
we shall show that our construction actually provides explicit solutions
to other problems: 
the direct acyclic decomposition of cycle Kuramoto networks,
and the self-intersection of a tropical hypersurface.

\subsection{Direct acyclic decomposition of cycle networks}\label{sec:dag}

In the recent work~\cite{Chen2019Directed}, a general scheme is proposed to decompose
a Kuramoto network into smaller subnetworks supported by direct acyclic graphs
while preserving certain properties of the synchronization configurations.
This scheme utilize the geometric properties of the adjacent polytope.
Indeed, the subnetworks are in one-to-one correspondence with the
facets of the adjacency polytope.

In this context, the constructions proposed in this paper
provide two important improvement to that decomposition scheme.
First, the regular unimodular triangulation of the adjacency polytope $P_{C_N}$
gives rise to a significant refinement for the direct acyclic decomposition scheme
which will decompose a cycle network into ``primitive'' subnetworks
that can be analyzed easily and exactly.
This was not possible for even $N$ values with the original decomposition scheme.
Second, as the starting system induced by the adjacency polytope homotopy~\eqref{equ:homotopy}
can be solve easily and efficiently.
Finally, since the explicit formula for the generic root count is known (\cref{sec:ap}),
the number of solution paths induced by the homotopy proposed in~\cref{sec:homotopy} 
matches the generic root count exactly,
and thus will not produce extraneous solution paths in the generic situation.
This feature was not established in the original decomposition scheme
and the associated homotopy construction.

To see the regular unimodular triangulation described in~\cref{sec:triangulation}
give rises to a decomposition of the Kuramoto network, 
we first define the subnetwork corresponds to a cell.

\begin{definition}[Directed acyclic subnetwork]
    Let $\Delta_N$ be the regular unimodular triangulation of $P_{C_N}$
    defined in~\eqref{equ:triangulation-even} and~\eqref{equ:triangulation-odd}.
    For each cell $T \in \Delta_N$, we define the \emph{directed acyclic subnetwork} 
    associated with $T$ to be the graph $(\{0,\dots,N-1\},\edges(T))$ where
    \[
        \edges(T) = \{ (i,j) \in \edges(C_N) \mid \bolde_i - \bolde_j \in T \}.
    \]
\end{definition}

This is a refinement of the definition given in~\cite{Chen2019Directed}
where subnetworks correspond to facets of the adjacency polytope.
In contrast, subnetworks defined above come from a triangulation
which, in the case of even $N$ values, are associated with simplices in the facets of $P_{C_N}$.

As established in~\cite{Chen2019Directed}, such a subnetwork associated with a cell
is always an \emph{acyclic} graph which justifies its name (directed acyclic subnetwork).
Moreover, such subnetworks are of the simplest possible form
known as ``primitive'' subnetworks.

\begin{definition}[Primitive directed acyclic subnetwork]
    A subnetwork associated with a cell, as defined above, 
    is said to be \emph{primitive} if it contains exactly $n = N-1$ directed edges.
\end{definition}

Primitive subnetworks may be considered the building blocks of a Kuramoto network
as the are the smallest directed acyclic subnetworks that weakly connected and contains all nodes.
More importantly, its synchronization configurations can be analyzed easily and exactly.
Note that since each cell $T \in \Delta_N$ is a simplex of dimension $n$
and contains exactly $n$ nonzero points of the form $\bolde_i - \bolde_j$ for $i \ne j$,
we can see the induced subnetwork must be primitive.

\begin{proposition}
    Let $\Delta_N$ be the regular unimodular triangulation of $P_{C_N}$.
    For each cell $T \in \Delta_N$, the associated directed acyclic subnetwork 
    is primitive.
\end{proposition}

\Cref{fig:dac-cycle4} shows the direct acyclic subnetworks of a cycle network with 4 nodes
induced by the triangulation $\Delta_N$.
All subnetworks are primitive.
In contrast, the original decomposition scheme, shown in~\cref{fig:dac-cycle4-old}, 
produces subnetworks that are not primitive.

\begin{figure}[ht]
    \centering
    \begin{tabular}{ccccc}
        \includegraphics[width=0.15\textwidth]{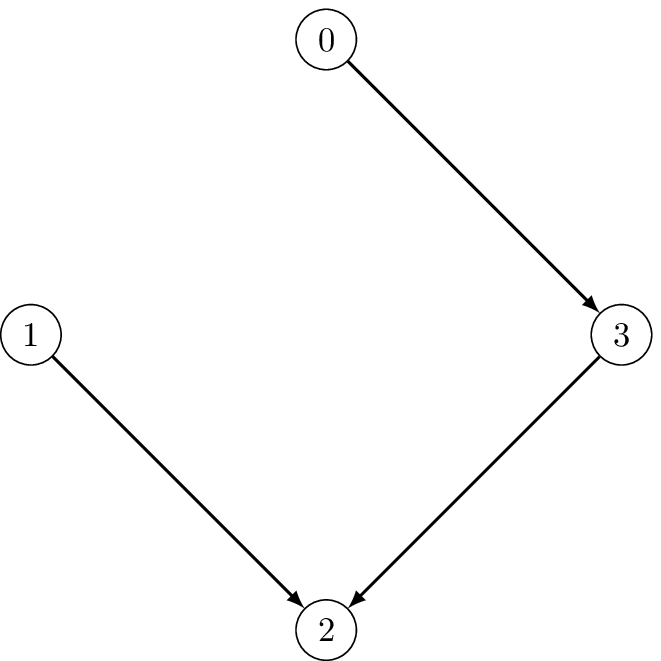} &  
        \includegraphics[width=0.15\textwidth]{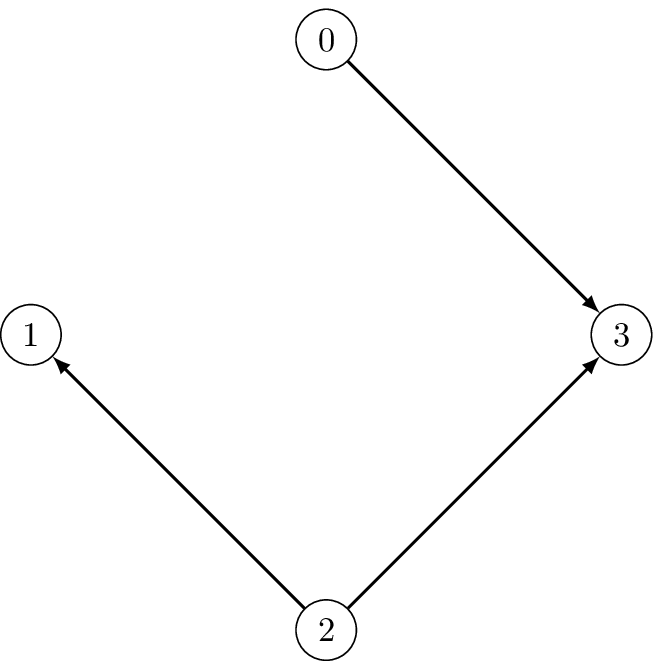} &  
        \includegraphics[width=0.15\textwidth]{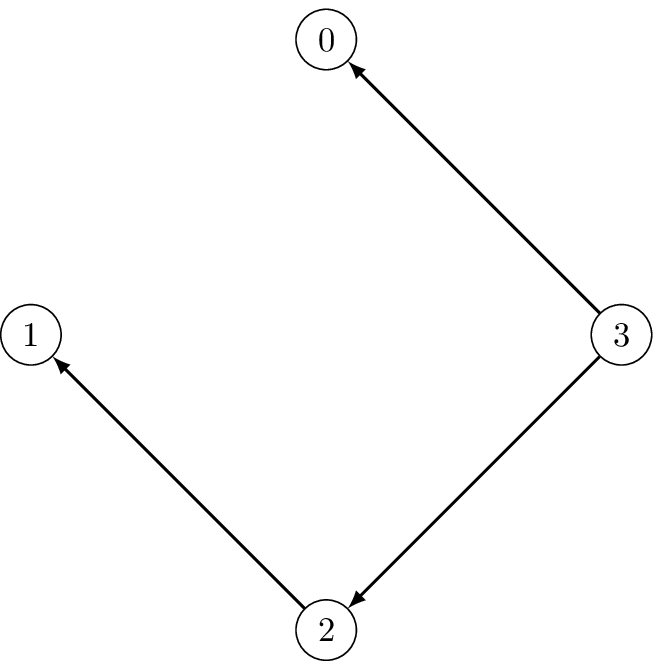} &  
        \includegraphics[width=0.15\textwidth]{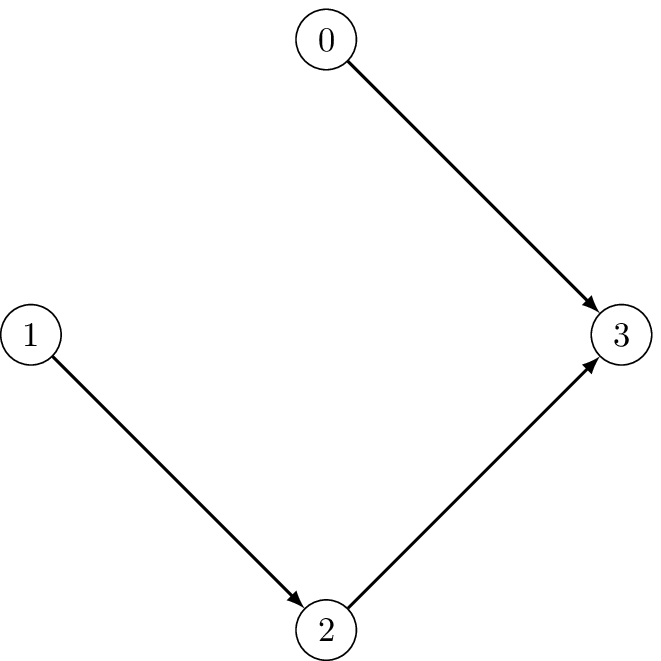} &  
        \\[1ex]
        \includegraphics[width=0.15\textwidth]{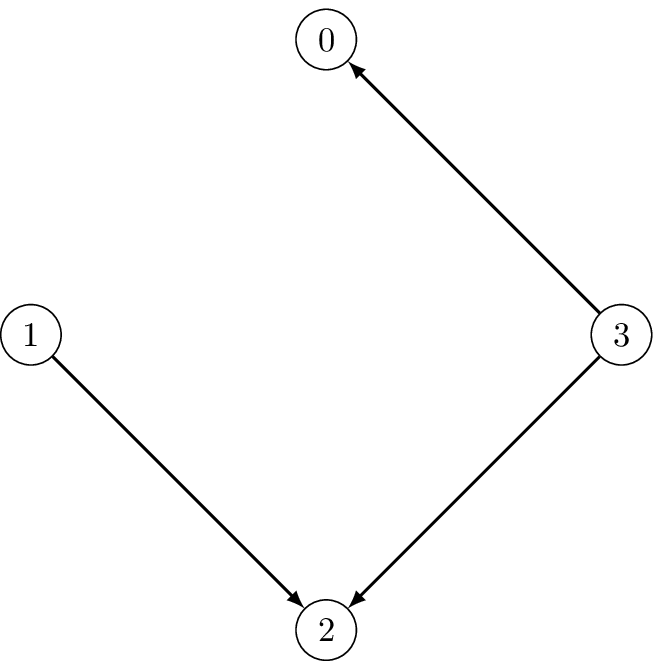} &  
        \includegraphics[width=0.15\textwidth]{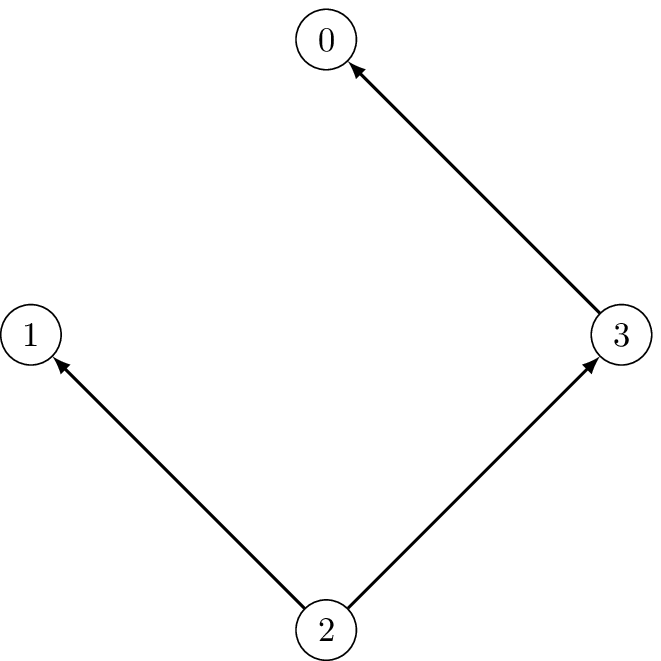} &  
        \includegraphics[width=0.15\textwidth]{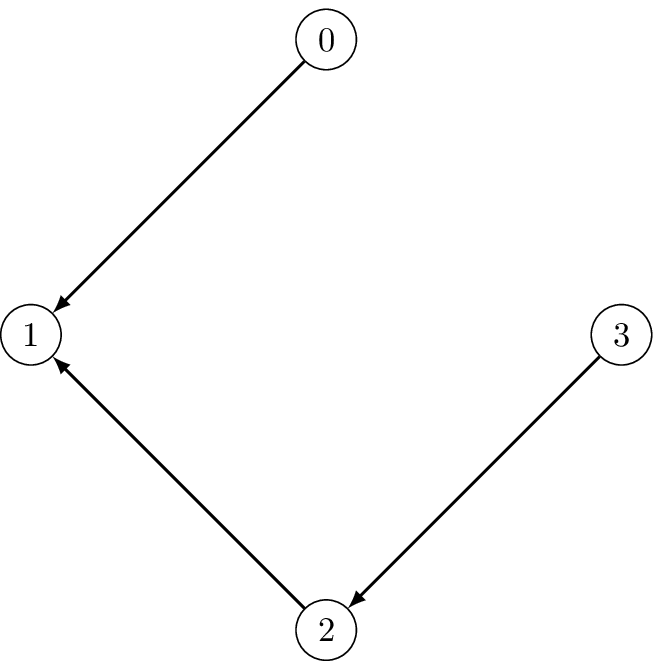} &  
        \includegraphics[width=0.15\textwidth]{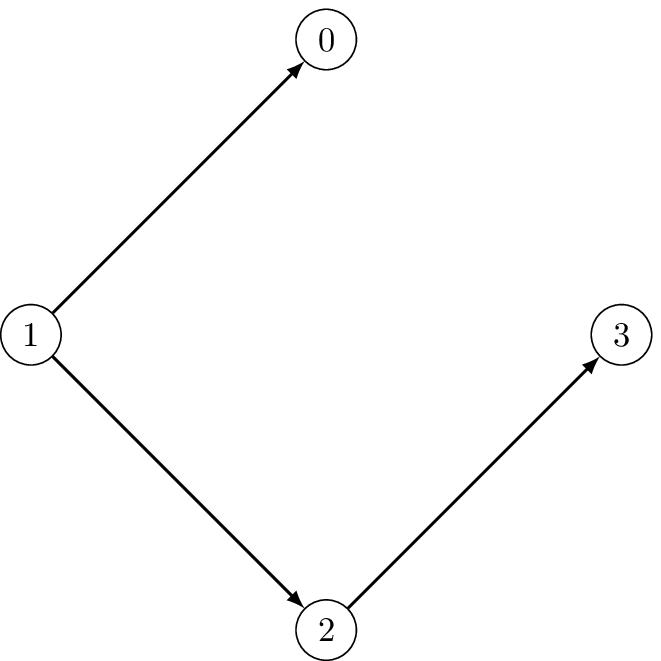} &  
        \\[1ex]
        \includegraphics[width=0.15\textwidth]{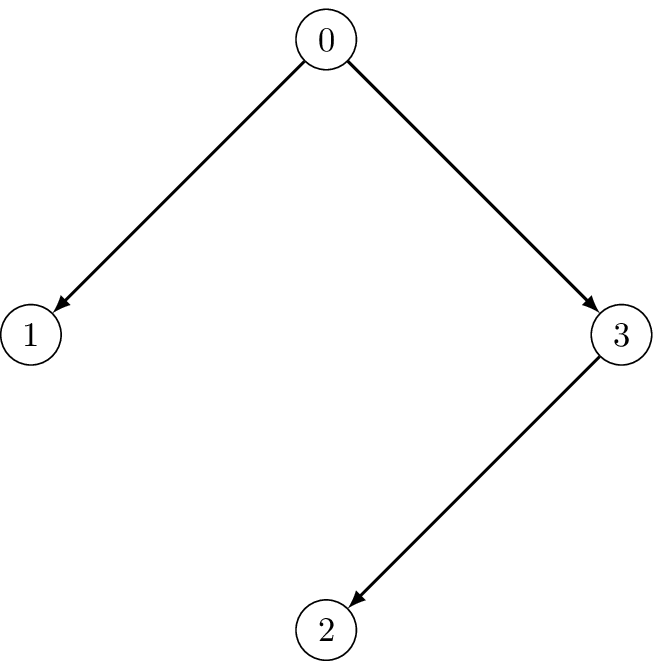} &  
        \includegraphics[width=0.15\textwidth]{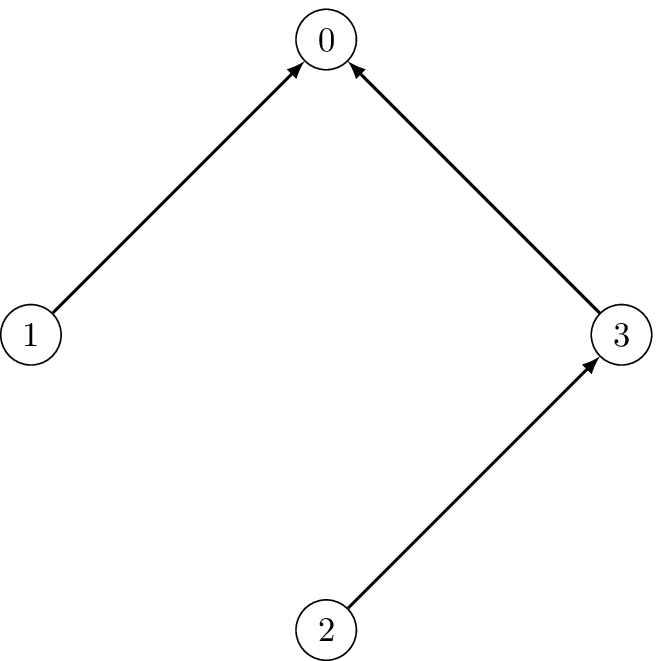} &  
        \includegraphics[width=0.15\textwidth]{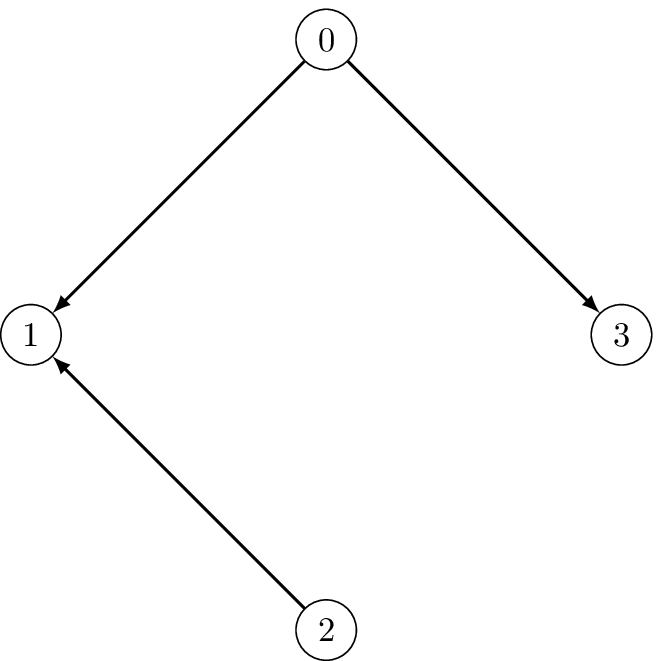} &  
        \includegraphics[width=0.15\textwidth]{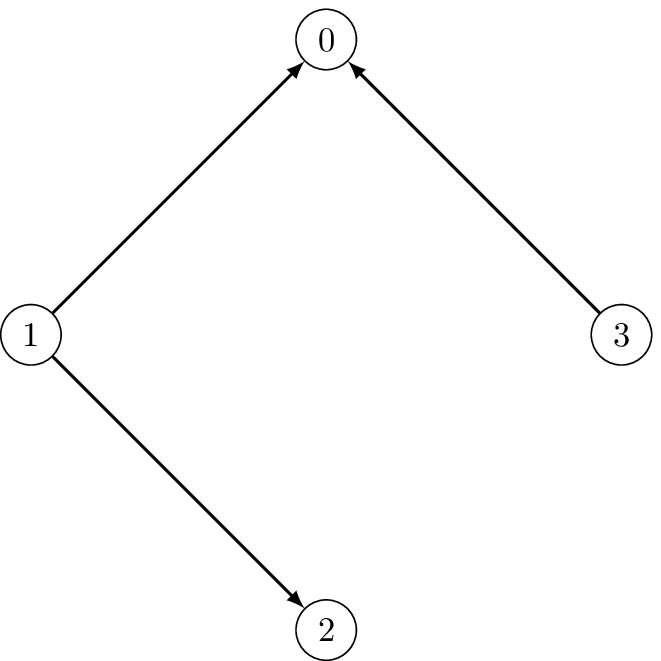} &  
    \end{tabular}
    \caption{
        Direct acyclic subnetworks of a cycle network with 4 nodes
        induced by the most refined decomposition scheme developed in this paper.
        Every subnetwork is \emph{primitive}.
        This is to be compared with the original coarser decomposition scheme
        shown in~\cref{fig:dac-cycle4-old}.
    }
    \label{fig:dac-cycle4}
\end{figure}

\begin{figure}[ht]
    \centering
    \begin{tabular}{cccc}
        \includegraphics[width=0.22\textwidth]{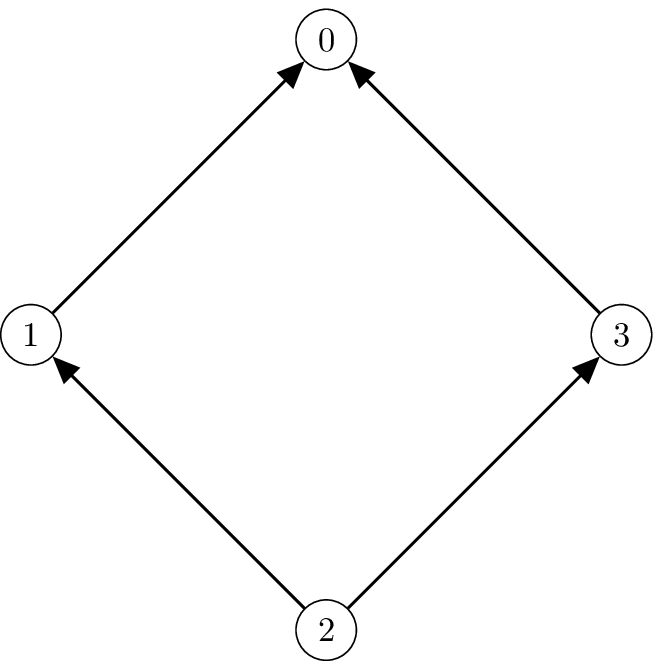} &  
        \includegraphics[width=0.22\textwidth]{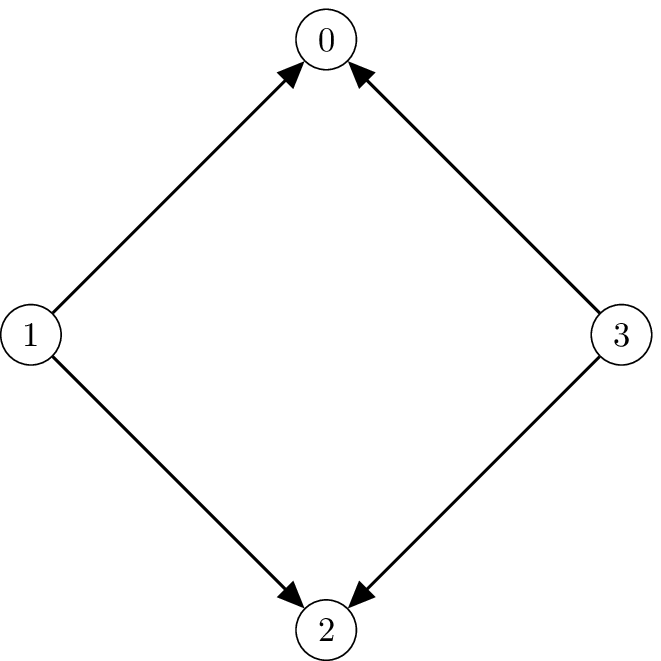} &  
        \includegraphics[width=0.22\textwidth]{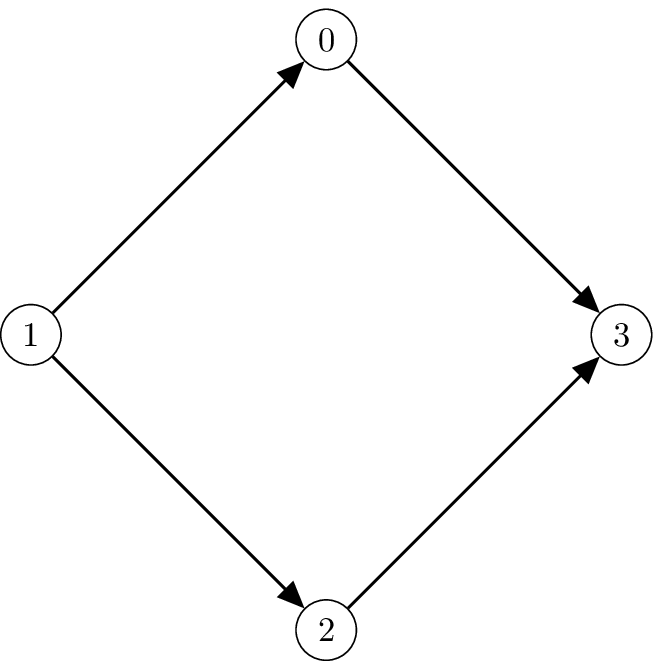} &  
        \\[1ex]
        \includegraphics[width=0.22\textwidth]{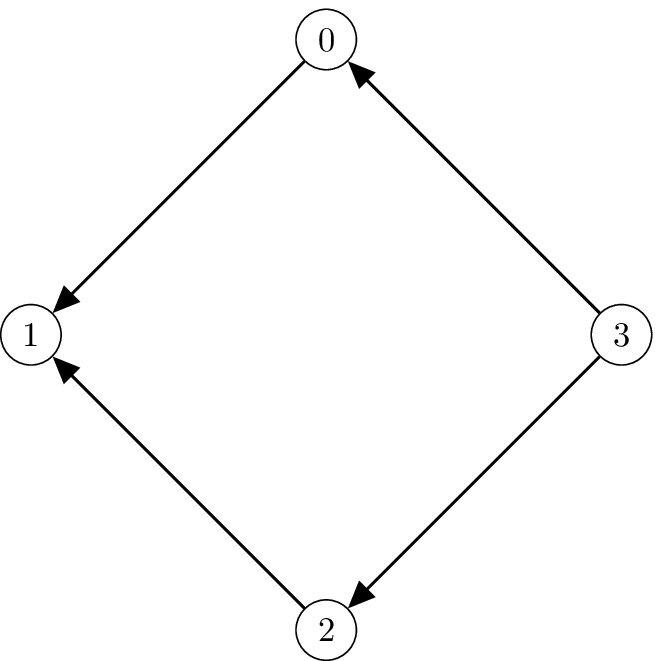} &  
        \includegraphics[width=0.22\textwidth]{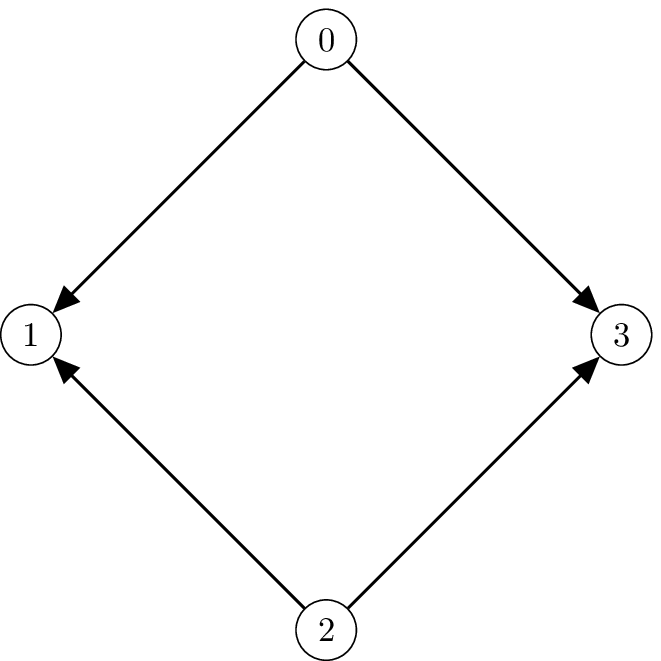} &  
        \includegraphics[width=0.22\textwidth]{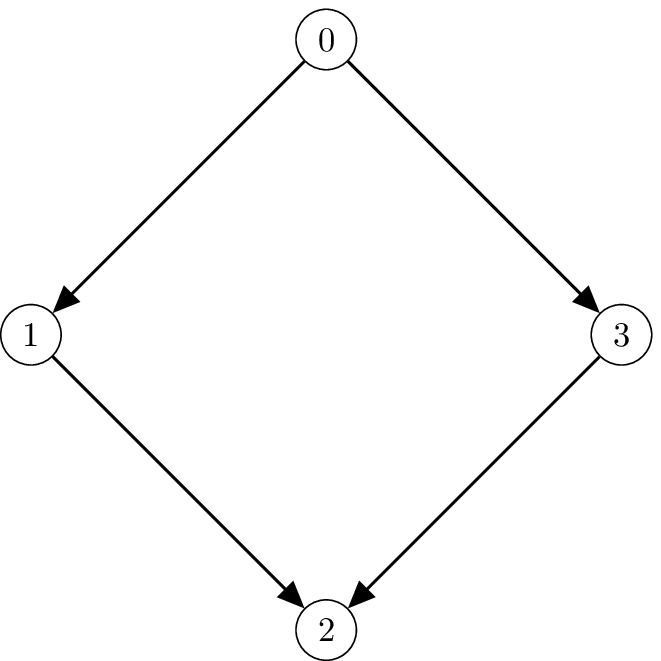} &  
    \end{tabular}
    \caption{
        Direct acyclic subnetworks of a cycle network with 4 nodes
        induced by the coarser decomposition scheme
        originally proposed in~\cite{Chen2019Directed}.
        None of the subnetworks are primitive.
    }
    \label{fig:dac-cycle4-old}
\end{figure}

\subsection{A tropical interpretation}\label{sec:tropical}

Even though it was not stated explicitly, the procedure that resulted in
the adjacency polytope homotopy~\eqref{equ:homotopy} is actually rooted from
tropical algebraic geometry~\cite{MaclaganSturmfels2015Introduction}.
In this section, we provide the interpretation from the tropical viewpoint.

Recall that we started with the unmixed form of 
the algebraic synchronization equation~\eqref{equ:kuramoto-rand}.
If we consider the valuation on the field of coefficients given by
\begin{align}\label{equ:valuation}
    \operatorname{val}(c_i^R) &= 0
    &&\text{and} &
    \operatorname{val}(a_{ijk}^R) &=
    \begin{cases}
        2 & \text{if $N$ is even and } (i,j) = (1,0) \\
        1 & \text{otherwise},
    \end{cases}
\end{align}
which mirrors the choices of the weights given in~\eqref{equ:weights-even} and~\eqref{equ:weights-odd},
then the tropicalization of the $n$ polynomials in~\eqref{equ:kuramoto-ode}
are identical and they define a common tropical hypersurface.
The main results developed in~\cref{sec:triangulation} can thus be interpreted tropically:
The valuation defined above induces the simplest (stable) self-intersection.

\begin{proposition}
    Let $h = \operatorname{trop}(F_{C_N,k}^R)$ for $k = 1,\dots,n$
    be the tropicalization of~\eqref{equ:kuramoto-rand} with respect to
    the valuation given in~\eqref{equ:valuation}.
    Then the tropical hypersurface defined by $h$ has exactly
    $\apbound$ self-intersection points,
    and each intersection is of multiplicity one.
\end{proposition}

As discussed in~\cref{rmk:no-01}, the special choice of the valuation~\eqref{equ:valuation}
is an important condition for this result to hold.
Using only 0-1 valuations, for example, will not produce self-intersections
with multiplicity one.
One of the key contribution of this paper is the explicit formula for
these self-intersection points.
These tropical self-intersection points are precisely the tropicalizations 
of the curves defined by~\eqref{equ:homotopy}.

\subsection{Equivalence of interpretations}

With the above interpretations, we have the equivalence of three important class
of problems from different context, as shown in~\cref{tab:dictionary}.
This connection, especially the connection between regular subdivision of adjacency polytope
and direct acyclic decomposition was first proposed in~\cite{Chen2019Directed}.
In this paper, we further refine this idea and provided explicit answers 
for problems in the bottom row of~\cref{tab:dictionary} in the cases of cycle networks.

\begin{table}[ht]
    \centering
    \begin{tabular}{ccc}
        \toprule
        Adjacency polytope & Kuramoto network & Tropical hypersurface\\ \toprule
        Regular subdivision & 
        Directed acyclic decomposition &
        Stable self-intersections\\ \midrule
        \makecell{Regular unimodular\\triangulation} & 
        \makecell{Directed acyclic decomposition\\into primitive subnetworks} &
        \makecell{Stable self-intersections\\with multiplicity one} 
        \\
        \bottomrule\\
    \end{tabular}
    \caption{
        The 3-way dictionary that translates equivalent concepts 
        among the three different points of view.
    }
    \label{tab:dictionary}
\end{table}

\section{Conclusions}\label{sec:conclusion}

Following the volume computation result in~\cite{ChenDavisMehta2018},
this paper aims to deepen the geometric understanding of adjacency polytopes
associated to a cycle Kuramoto network and use these geometric information to 
explore three aspects of Kuramoto equations:
\begin{enumerate}
    \item 
        To create an efficient polyhedral-like homotopy for
        solving Kuramoto equations;
    \item
        To explicitly describe direct acyclic decompositions
        of Kuramoto networks into \emph{primitive} subnetworks; and
    \item
        To understand the stable intersections of the tropical hypersurfaces
        defined by Kuramoto equations.
\end{enumerate}

First, we derived the explicit formula for a regular unimodular triangulation 
of the adjacency polytope $P_{C_N}$ associated to a cycle graph of $N$ nodes for any $N > 2$.
This greatly strengthens the results from~\cite{ChenDavisMehta2018} where
only the normalized volume of $P_{C_N}$ is known.

Then, using this regular unimodular triangulation, 
we develop a homotopy continuation algorithm based on the
well established polyhedral homotopy method yet has the distinct advantage
that it completely sidesteps the costly mixed volume/cells computation step.
This homotopy is also a significant improvement over the direct acyclic homotopy
proposed in~\cite{Chen2019Directed} since it deforms the Kuramoto system
into simplest possible subsystems each having a unique solution.
From the computational viewpoint, the proposed homotopy also offers important advantages
in numerical conditions, efficiency, and scalability as discussed in
\cref{rmk:numerical}.

The third contribution of this paper is a significantly refined version 
of the direct acyclic decomposition scheme originally proposed in~\cite{Chen2019Directed}.
The regular unimodular triangulation proposed here induces a decomposition
of a cycle Kuramoto network into the smallest possible components
known as primitive subnetworks.
Primitive subnetworks are of great value since they each have a unique 
complex synchronization configuration which can be computed easily and efficiently.
This is to be compared with the situation of the original decomposition scheme 
where the resulting subnetworks, in general, may not be primitive.

Finally, interpreted in the context of tropical geometry,
our result provides explicit formula for all stable intersections of
the tropical hypersurfaces defined by the unmixed form of the Kuramoto system
under a special choice of the valuation.
The induced tropical intersections are particularly nice
as we shown that every intersection point is of multiplicity 1.



 \bibliographystyle{siamplain}
 \bibliography{kuramoto,bkk,software}

\begin{thebibliography}{10}

\bibitem{ArdilaEtAlRoots}
{\sc F.~Ardila, M.~Beck, S.~Ho\c{s}ten, J.~Pfeifle, and K.~Seashore}, {\em Root
  polytopes and growth series of root lattices}, SIAM J. Discrete Math., 25
  (2011), pp.~360--378, \url{https://doi.org/10.1137/090749293},
  \url{https://doi.org/10.1137/090749293}.

\bibitem{Baillieul1982}
{\sc J.~Baillieul and C.~Byrnes}, {\em {Geometric critical point analysis of
  lossless power system models}}, IEEE Transactions on Circuits and Systems, 29
  (1982), pp.~724--737, \url{https://doi.org/10.1109/TCS.1982.1085093}.

\bibitem{Casetti:June2003:0022-4715:1091}
{\sc L.~Casetti, M.~Pettini, and E.~G.~D. Cohen}, {\em Phase transitions and
  topology changes in configuration space}, Journal of Statistical Physics, 111
  (June 2003), pp.~1091--1123(33),
  \url{http://www.ingentaconnect.com/content/klu/joss/2003/00000111/F0020005/00462874}.

\bibitem{Chen2017Unmixing}
{\sc T.~Chen}, {\em Unmixing the mixed volume computation}, arXiv:1703.01684
  [math],  (2017), \url{http://arxiv.org/abs/1703.01684}.

\bibitem{Chen2018KAPCycle}
{\sc T.~Chen}, {\em chentianran/kap-cycle: Version 1.0.1}, Oct. 2018,
  \url{https://doi.org/10.5281/zenodo.1439501},
  \url{https://doi.org/10.5281/zenodo.1439501}.

\bibitem{Chen2019Directed}
{\sc T.~Chen}, {\em {Directed acyclic decomposition of Kuramoto equations}},
  Chaos: An Interdisciplinary Journal of Nonlinear Science, 29 (2019),
  p.~093101, \url{https://doi.org/10.1063/1.5097826},
  \url{http://aip.scitation.org/doi/10.1063/1.5097826}.

\bibitem{ChenDavisMehta2018}
{\sc T.~Chen, R.~Davis, and D.~Mehta}, {\em {Counting equilibria of the
  Kuramoto model using birationally invariant intersection index}}, SIAM
  Journal on Applied Algebra and Geometry, 2 (2018), pp.~489--507,
  \url{https://doi.org/10.1137/17M1145665},
  \url{https://epubs.siam.org/doi/10.1137/17M1145665}.

\bibitem{ChenMehtaNiemerg2016}
{\sc T.~Chen, J.~Marecek, D.~Mehta, and M.~Niemerg}, {\em A network topology
  dependent upper bound on the number of equilibria of the kuramoto model},
  arXiv preprint arXiv:1603.05905,  (2016).

\bibitem{TriangulationsBook}
{\sc J.~A. De~Loera, J.~Rambau, and F.~Santos}, {\em Triangulations}, vol.~25
  of Algorithms and Computation in Mathematics, Springer-Verlag, Berlin, 2010,
  \url{https://doi.org/10.1007/978-3-642-12971-1},
  \url{https://doi-org.proxy1.cl.msu.edu/10.1007/978-3-642-12971-1}.
\newblock Structures for algorithms and applications.

\bibitem{delabays2016multistability}
{\sc R.~Delabays, T.~Coletta, and P.~Jacquod}, {\em Multistability of
  phase-locking and topological winding numbers in locally coupled kuramoto
  models on single-loop networks}, Journal of Mathematical Physics, 57 (2016),
  p.~032701.

\bibitem{delabays2017multistability}
{\sc R.~Delabays, T.~Coletta, and P.~Jacquod}, {\em Multistability of
  phase-locking in equal-frequency kuramoto models on planar graphs}, Journal
  of Mathematical Physics, 58 (2017), p.~032703.

\bibitem{DelucchiHoessly2016Fundamental}
{\sc E.~Delucchi and L.~Hoessly}, {\em {Fundamental polytopes of metric trees
  via parallel connections of matroids}},  (2016),
  \url{http://arxiv.org/abs/1612.05534},
  \url{https://arxiv.org/abs/1612.05534}.

\bibitem{dorfler_synchronization_2014}
{\sc F.~D{\"{o}}rfler and F.~Bullo}, {\em {Synchronization in complex networks
  of phase oscillators: A survey}}, Automatica, 50 (2014), pp.~1539--1564,
  \url{https://doi.org/10.1016/j.automatica.2014.04.012}.

\bibitem{fulton_introduction_1993}
{\sc W.~Fulton}, {\em Introduction to toric varieties}, no.~131, Princeton
  University Press, 1993.

\bibitem{GaoLiVerscheldeWu2000Balancing}
{\sc T.~Gao, T.~Y. Li, J.~Verschelde, and M.~Wu}, {\em {Balancing the lifting
  values to improve the numerical stability of polyhedral homotopy continuation
  methods}}, Applied Mathematics and Computation, 114 (2000), pp.~233--247,
  \url{https://doi.org/10.1016/S0096-3003(99)00115-0},
  \url{http://www.sciencedirect.com/science/article/pii/S0096300399001150}.

\bibitem{M2}
{\sc D.~R. Grayson and M.~E. Stillman}, {\em Macaulay2, a software system for
  research in algebraic geometry}.
\newblock Available at \url{http://www.math.uiuc.edu/Macaulay2/}.

\bibitem{SymmetricEdge2019}
{\sc A.~Higashitani, K.~Jochemko, and M.~Micha\l~ek}, {\em Arithmetic aspects
  of symmetric edge polytopes}, Mathematika, 65 (2019), pp.~763--784,
  \url{https://doi.org/10.1112/s0025579319000147},
  \url{https://doi.org/10.1112/s0025579319000147}.

\bibitem{InterlacingEhrhart}
{\sc A.~Higashitani, M.~Kummer, and M.~Micha\l~ek}, {\em Interlacing {E}hrhart
  polynomials of reflexive polytopes}, Selecta Math. (N.S.), 23 (2017),
  pp.~2977--2998, \url{https://doi.org/10.1007/s00029-017-0350-6},
  \url{https://doi.org/10.1007/s00029-017-0350-6}.

\bibitem{Higashitani2016Interlacing}
{\sc A.~Higashitani, M.~Kummer, and M.~Micha{\l}ek}, {\em {Interlacing Ehrhart
  Polynomials of Reflexive Polytopes}},  (2016),
  \url{https://doi.org/10.1007/s00029-017-0350-6},
  \url{http://arxiv.org/abs/1612.07538
  http://dx.doi.org/10.1007/s00029-017-0350-6},
  \url{https://arxiv.org/abs/1612.07538}.

\bibitem{huber_polyhedral_1995}
{\sc B.~Huber and B.~Sturmfels}, {\em A polyhedral method for solving sparse
  polynomial systems}, Mathematics of Computation, 64 (1995), pp.~1541--1555,
  \url{https://doi.org/10.1090/S0025-5718-1995-1297471-4}.

\bibitem{Hughes:2012hg}
{\sc C.~Hughes, D.~Mehta, and J.-I. Skullerud}, {\em Enumerating gribov copies
  on the lattice}, Annals of Physics, 331 (2013), pp.~188 -- 215,
  \url{https://doi.org/https://doi.org/10.1016/j.aop.2012.12.011},
  \url{http://www.sciencedirect.com/science/article/pii/S0003491613000079}.

\bibitem{hughes2014inversion}
{\sc C.~Hughes, D.~Mehta, and D.~J. Wales}, {\em An inversion-relaxation
  approach for sampling stationary points of spin model hamiltonians}, The
  Journal of chemical physics, 140 (2014), p.~194104.

\bibitem{kastner2011stationary}
{\sc M.~Kastner}, {\em Stationary-point approach to the phase transition of the
  classical xy chain with power-law interactions}, Physical Review E, 83
  (2011), p.~031114.

\bibitem{KavehKhovanskii2012Newton}
{\sc K.~Kaveh and A.~Khovanskii}, {\em {Newton-Okounkov} bodies, semigroups of
  integral points, graded algebras and intersection theory}, Annals of
  Mathematics, 176 (2012), pp.~925--978,
  \url{https://doi.org/10.4007/annals.2012.176.2.5},
  \url{http://arxiv.org/abs/0904.3350
  http://annals.math.princeton.edu/2012/176-2/p05},
  \url{https://arxiv.org/abs/0904.3350}.

\bibitem{KavehKhovanskii2010Mixed}
{\sc K.~Kaveh and A.~G. Khovanskii}, {\em Mixed volume and an extension of
  theory of divisors}, Moscow Mathematical Journal, 10 (2010), pp.~343--375,
  \url{http://www.ams.org/distribution/mmj/vol10-2-2010/kaveh-khovanskii.pdf},
  \url{https://arxiv.org/abs/0812.0433}.

\bibitem{Kuramoto1975}
{\sc Y.~Kuramoto}, {\em {Self-entrainment of a population of coupled non-linear
  oscillators}}, Lecture Notes in Physics, Springer Berlin Heidelberg, 1975,
  pp.~420--422, \url{http://link.springer.com/chapter/10.1007/BFb0013365}.

\bibitem{Kuramoto2012}
{\sc Y.~Kuramoto}, {\em {Chemical Oscillations, Waves, and Turbulence}},
  Springer Science {\&} Business Media, dec 2012,
  \url{https://books.google.com/books?id=tcTyCAAAQBAJ}.

\bibitem{Li2003Numerical}
{\sc T.-Y. Li}, {\em {Numerical Solution of Polynomial Systems by Homotopy
  Continuation}}, in Handbook of Numerical Analysis: Special Volume:
  Foundations of Computational Mathematics, P.~G. Ciarlet, ed., vol.~11,
  North-Holland, 2003, p.~470,
  \url{https://doi.org/10.1016/S1570-8659(02)11004-0}.

\bibitem{MaclaganSturmfels2015Introduction}
{\sc D.~Maclagan and B.~Sturmfels}, {\em Introduction to Tropical Geometry},
  Graduate Studies in Mathematics, American Mathematical Society, 2015,
  \url{https://books.google.com/books?id=zFsoCAAAQBAJ}.

\bibitem{manik2016cycle}
{\sc D.~Manik, M.~Timme, and D.~Witthaut}, {\em {Cycle flows and multistability
  in oscillatory networks}}, Chaos, 27 (2017), p.~083123,
  \url{https://doi.org/10.1063/1.4994177},
  \url{https://arxiv.org/abs/1611.09825}.

\bibitem{Matsui2011Roots}
{\sc T.~Matsui, A.~Higashitani, Y.~Nagazawa, H.~Ohsugi, and T.~Hibi}, {\em
  {Roots of Ehrhart polynomials arising from graphs}}, Journal of Algebraic
  Combinatorics, 34 (2011), pp.~721--749,
  \url{https://doi.org/10.1007/s10801-011-0290-8},
  \url{http://link.springer.com/10.1007/s10801-011-0290-8}.

\bibitem{Mehta2015Algebraic}
{\sc D.~Mehta, N.~S. Daleo, F.~Dörfler, and J.~D. Hauenstein}, {\em Algebraic
  geometrization of the kuramoto model: Equilibria and stability analysis},
  Chaos: An Interdisciplinary Journal of Nonlinear Science, 25 (2015),
  p.~053103, \url{https://doi.org/10.1063/1.4919696},
  \url{https://doi.org/10.1063/1.4919696},
  \url{https://arxiv.org/abs/https://doi.org/10.1063/1.4919696}.

\bibitem{Mehta:2013iea}
{\sc D.~Mehta, C.~Hughes, M.~Schr\"ock, and D.~Wales}, {\em Potential energy
  landscapes for the 2d xy model: Minima, transition states, and pathways}, The
  Journal of Chemical Physics, 139 (2013), p.~194503.

\bibitem{mehta2011stationary}
{\sc D.~Mehta and M.~Kastner}, {\em Stationary point analysis of the
  one-dimensional lattice landau gauge fixing functional, aka random phase xy
  hamiltonian}, Annals of Physics, 326 (2011), pp.~1425--1440,
  \url{https://doi.org/10.1016/j.aop.2010.12.016},
  \url{https://arxiv.org/abs/1010.5335}.

\bibitem{Nerattini:2012pi}
{\sc R.~Nerattini, M.~Kastner, D.~Mehta, and L.~Casetti}, {\em Exploring the
  energy landscape of $xy$ models}, Phys. Rev. E, 87 (2013), p.~032140,
  \url{https://doi.org/10.1103/PhysRevE.87.032140},
  \url{https://link.aps.org/doi/10.1103/PhysRevE.87.032140}.

\bibitem{Nill}
{\sc B.~Nill}, {\em Classification of pseudo-symmetric simplicial reflexive
  polytopes}, in Algebraic and geometric combinatorics, vol.~423 of Contemp.
  Math., Amer. Math. Soc., Providence, RI, 2006, pp.~269--282,
  \url{https://doi.org/10.1090/conm/423/08082},
  \url{https://doi-org.proxy1.cl.msu.edu/10.1090/conm/423/08082}.

\bibitem{ochab2010synchronization}
{\sc J.~Ochab and P.~Gora}, {\em Synchronization of coupled oscillators in a
  local one-dimensional kuramoto model}, Acta Physica Polonica. Series B,
  Proceedings Supplement, 3 (2010), pp.~453--462.

\bibitem{SymmetricConfigurations}
{\sc H.~Ohsugi and T.~Hibi}, {\em Centrally symmetric configurations of integer
  matrices}, Nagoya Math. J., 216 (2014), pp.~153--170,
  \url{https://doi.org/10.1215/00277630-2857555},
  \url{https://doi.org/10.1215/00277630-2857555}.

\bibitem{SmoothFanoEhrhart}
{\sc H.~Ohsugi and K.~Shibata}, {\em Smooth {F}ano polytopes whose {E}hrhart
  polynomial has a root with large real part}, Discrete Comput. Geom., 47
  (2012), pp.~624--628, \url{https://doi.org/10.1007/s00454-012-9395-7},
  \url{https://doi.org/10.1007/s00454-012-9395-7}.

\bibitem{Postnikov}
{\sc A.~Postnikov}, {\em Permutohedra, associahedra, and beyond}, Int. Math.
  Res. Not. IMRN,  (2009), pp.~1026--1106,
  \url{https://doi.org/10.1093/imrn/rnn153},
  \url{https://doi.org/10.1093/imrn/rnn153}.

\bibitem{Sommese2005}
{\sc A.~J. Sommese and C.~W. Wampler}, {\em The Numerical Solution of Systems
  of Polynomials Arising in Engineering and Science}, WORLD SCIENTIFIC, mar
  2005, \url{https://doi.org/10.1142/5763},
  \url{http://www.worldscientific.com/worldscibooks/10.1142/5763}.

\bibitem{xi2017synchronization}
{\sc K.~Xi, J.~L. Dubbeldam, and H.~X. Lin}, {\em Synchronization of cyclic
  power grids: Equilibria and stability of the synchronous state}, Chaos: An
  Interdisciplinary Journal of Nonlinear Science, 27 (2017), p.~013109.

\bibitem{xin2016analytical}
{\sc X.~Xin, T.~Kikkawa, and Y.~Liu}, {\em Analytical solutions of equilibrium
  points of the standard kuramoto model: 3 and 4 oscillators}, in American
  Control Conference (ACC), 2016, IEEE, 2016, pp.~2447--2452.

\end{thebibliography}

\end{document}